\documentclass[final,leqno]{siamltex}
\usepackage{mathbbol}
\usepackage{mathrsfs}
\usepackage{amsfonts}
\usepackage{stmaryrd}
\usepackage{amsmath}
\usepackage{amssymb}
\usepackage{multirow}
\usepackage{caption}
\usepackage{tabularx}
\usepackage{booktabs}
\usepackage{cite}
\usepackage{verbatim}
\usepackage{color}
\usepackage{graphicx}
\usepackage{subfigure}
\usepackage{hyperref}
\usepackage[all,cmtip]{xy}
\usepackage{epsfig,epstopdf,color,bm}
\allowdisplaybreaks[4]

\usepackage{lipsum}
\usepackage{algorithmic}
\usepackage{cases}
\usepackage{bm}
\usepackage{amsopn}
\usepackage{tabularx}
\usepackage{booktabs}
\usepackage{cite}
\usepackage{verbatim}
\usepackage{color}
\usepackage{graphicx}
\usepackage{subfigure}

\newcommand{\p}{\partial}

\newcommand{\be}{\begin{equation}}
\newcommand{\ee}{\end{equation}}
\newcommand{\ba}{\begin{array}}
\newcommand{\ea}{\end{array}}
\newcommand{\bea}{\begin{eqnarray}}
\newcommand{\eea}{\end{eqnarray}}
\newcommand{\beas}{\begin{eqnarray*}}
\newcommand{\eeas}{\end{eqnarray*}}

\newtheorem{remark}{Remark}[section]
\newtheorem{example}{Example}[section]
\newtheorem{algorithm}{Algorithm}[section]

\def\R{{\mathbb R}}

\newcommand{\cE}{\mathcal E}

\renewcommand{\l}{\left}
\renewcommand{\r}{\right}

\renewcommand{\d}{\mathrm{d}}

\newcommand{\bn}{\mathbf{n}}

\newcommand\cV{{\mathcal V}}
\newcommand{\bX}{\mathbf{X}}

\title{Stable BDF time discretization of BGN-based parametric finite element methods for geometric flows}
\author{Wei Jiang\thanks{School of Mathematics and Statistics,
Wuhan University, Wuhan, 430072, China ({\tt jiangwei1007@whu.edu.cn}). This author's research was supported by the National Natural Science Foundation of China Nos. 12271414 and 11871384.}
    \and
Chunmei Su\thanks{Yau Mathematical Sciences Center, Tsinghua University, Beijing, 100084, China ({\tt sucm@tsinghua.edu.cn}).
This author's research was supported by National Key R\&D Program of China (2023YFA1008902) and the National Natural Science Foundation of China No. 12201342.}
\and
Ganghui Zhang\thanks{Yau Mathematical Sciences Center, Tsinghua University, Beijing, 100084, China({\tt gh-zhang19@mails.tsinghua.edu.cn}).}}

\begin{document}

\maketitle

\begin{abstract}
We propose a novel class of temporal high-order parametric finite element methods for solving a wide range of geometric flows of curves and surfaces. By incorporating the backward differentiation formulae (BDF) for time discretization into the BGN formulation, originally proposed by Barrett, Garcke, and N\"urnberg (J. Comput. Phys., 222 (2007), pp.~441--467), we successfully develop high-order BGN/BDF$k$ schemes. The proposed BGN/BDF$k$ schemes not only retain almost all the advantages of the classical first-order BGN scheme such as computational efficiency and good mesh quality, but also exhibit the desired $k$th-order temporal accuracy in terms of shape metrics, ranging from second-order to fourth-order accuracy. Furthermore, we validate the performance of our proposed BGN/BDF$k$ schemes through extensive numerical examples, demonstrating their high-order temporal accuracy for various types of geometric flows while maintaining good mesh quality throughout the evolution.
\end{abstract}


\section{Introduction}

Geometric flows, which are also known as geometric PDEs, have been a subject of significant interest in the past several decades. These flows involve the evolution of a geometric shape from one form to another and have found applications in various fields, such as grain boundary motion~\cite{Mullins1956},
solid-state dewetting~\cite{Zhao-Jiang-Wang-Bao,Zhao-Jiang-Bao2021},  image processing~\cite{Cao,Sapiro}, biomembranes~\cite{Du05}, cellular automata~\cite{Ruuth99,Ruuth08}.
Numerical simulation has played a crucial role in this rapidly growing research area, aiding in the understanding of the underlying theory and guiding experimental investigations.

One of the most important numerical methods is the parametric finite element method, which was first proposed by Dziuk \cite{Dziuk1990} for  simulating  mean curvature flow of surfaces in three-dimensional space. Since then, this type of numerical schemes has been extensively employed for solving various types of geometric flows arising from science and engineering problems, including mean curvature flow \cite{Dziuk1994,KLL2019,KLL2020}, surface diffusion flow~\cite{Bao-Zhao,Jiang21}, Willmore flow \cite{Dziuk2008,Dziuk2002,KLL2021}, anisotropic flow~\cite{Dziuk1999,Bao-Jiang-Li}, and generalized mean curvature flow~\cite{BGN07B,Pei-Li,Tim2022}. For a more comprehensive overview of recent advances in the parametric finite element method, we recommend referring to the survey papers \cite{DDE2005,BGN20}.

One of the main difficulties in solving the geometric flows by parametric finite element methods is the mesh distortion problem. As time evolves, the nodes may cluster together and the mesh deteriorates, leading to unstable simulations and algorithm failures. To date, numerous efforts have been made to address this issue in the literature. One approach is the use of artificial mesh regularization methods to improve the distribution of mesh points during the evolution, as proposed by B\"ansch, Morin, and Nochetto \cite{Bansch-Morin-Nochetto}. Another approach is to include an additional tangential velocity functional in the equation to prevent numerical solutions from forming various instabilities, as studied by Mikula and \v{S}ev\v{c}ovi\v{c} \cite{Mikula-Sevcovic,Mikula-Sevcovic2004}. As demonstrated in \cite{Deckelnick-Dziuk,Elliott-Fritz}, DeTurck's trick, or more precisely, harmonic map meat flows can be utilized to maintain good mesh quality. In addition, a reparametrization technique based on discrete harmonic map \cite{Steinhilber2014} has been employed for remeshing the evolutionary polyhedron at each time step. Recently, Hu and Li \cite{Hu2022} have proposed a new evolving surface finite element method by introducing an artificial tangential velocity to improve the mesh quality for mean curvature flow and Willmore flow.

Instead of utilizing the mesh redistribution approach, the so-called BGN scheme, originally proposed by Barrett, Garcke, and N\"urnberg, is a parametric finite element method constructed based on a formulation which allows an intrinsic tangential velocity to ensure a good distribution of mesh points~\cite{BGN07A,BGN07B,BGN08C}.
Taking curve shortening flow (CSF) as an example, now we present the key idea of the BGN scheme and our approach to developing high-order BGN-based schemes. Let $\Gamma:=\Gamma(t)$ be a family of  simple closed curves embedded in the two-dimensional plane driven by CSF, i.e.,  the velocity is given by:
\begin{equation}\label{Geometric equation}
  \cV = -\kappa\mathbf{n}.
\end{equation}
Here $\kappa$ represents the curvature of the curve and we always assume that a circle has a positive curvature,  and $\mathbf{n}$ is the outward unit normal to $\Gamma$. Firstly, the curve $\Gamma(t)$ can be represented by a vector function $\mathbf{X}(\cdot,t):\mathbb{I}\rightarrow \R^2$, where $\mathbb{I}:=\mathbb{R}/\mathbb{Z}$ is the periodic interval $[0, 1]$. Then, the above equation~\eqref{Geometric equation} is rewritten
as the following BGN formulation~\cite{BGN07B}:
\begin{equation}
\begin{split}
	\p_t \mathbf{X}\cdot \mathbf{n} &=-\kappa,\\
  	\kappa \mathbf{n}&=-\p_{ss}\mathbf{X},
\end{split}
\label{eqn:weak}
\end{equation}
where $s:=s(t)$ represents the arc-length of $\Gamma(t)$. Compared to the original flow \eqref{Geometric equation}, a new variable $\kappa$ is introduced. This formulation is attractive since the normal velocity remains unchanged, preserving the shape of the evolving curve. On the other hand, the tangential velocity is not prescribed in \eqref{eqn:weak}, allowing for an intrinsic tangential movement.  Based on this formulation, a semi-discrete scheme was proposed \cite{BGN07B}
\begin{equation}\label{BGN1semi}
	\begin{split}
	\frac{\mathbf{X}^{m+1}-\mathbf{X}^{m}}{\tau}\cdot \mathbf{n}^m &=-\kappa^{m+1},\\
  	\kappa^{m+1} \mathbf{n}^m&=-\p_{s^ms^m}\mathbf{X}^{m+1},
	\end{split}
\end{equation}
where $\cdot^m$ represents the approximation solution at time $t_m:=m\tau$ with $\tau$ as the time step size, e.g., $\Gamma^{m}$ (which is parameterized by $\bX^{m}$) is an approximation of $\Gamma(t_m)$, and $\bn^{m}$, $s^m$ correspond to the  approximations of the unit outer normal vector and arc-length, respectively. By using the linear finite element method in space and taking the variational formulation over the polygon  $\Gamma^{m}$ on both sides of equations, Barrett, Garcke, and N\"urnberg derived the classical BGN scheme for CSF.
It was shown that the resulting fully discrete BGN scheme is well-posed and possesses several desirable properties, such as unconditional stability, energy dissipation, and asymptotic long-time mesh equal distribution~\cite{BGN07A,BGN07B,BGN20}.
However, it is worth noting that the fully discrete BGN scheme is limited to first-order accuracy in time, and developing a high-order BGN-based scheme remains a challenging task. Actually, naive high-order time discretizations based on the BGN formulation might cause mesh distortion problems \cite{Duan2023,Jiang23B}. For example, as pointed in \cite{Duan2023}, high-order time discretizations by the backward
differentiation formula, based on the BGN methods, become unstable probably (see Figures \ref{Fig:flower_shape_evo}-\ref{Fig:flower_geo_evo}).

In this paper, we present a series of temporal high-order schemes  using the backward differentiation formulae (BDF) based on the BGN formulation. The numerical instability problem pointed in \cite{Duan2023} can be avoided effectively by carefully selecting the prediction curve/surface to integrate on. Taking the BDF2 time discretization of \eqref{eqn:weak} as an example,
we consider the following semi-discrete in time scheme for CSF:
\begin{equation}	
\begin{split}
\frac{\frac{3}{2}\bX^{m+1}-2\mathbf{X}^{m}+\frac{1}{2}\mathbf{X}^{m-1}}{\tau} \widetilde{\mathbf{n}}^{m+1} &=-\kappa^{m+1},\\
  	\kappa^{m+1} \widetilde{\mathbf{n}}^{m+1}&=-\p_{\tilde{s}^{m+1}\tilde{s}^{m+1}}\mathbf{X}^{m+1}.
	\end{split}
\end{equation}
Here, we utilize a semi-implicit approach to avoid a fully implicit scheme. Thus a suitable explicit approximation $\widetilde{\Gamma}^{m+1}$ of $\Gamma(t_{m+1})$, on which $\widetilde{\bn}^{m+1}$, $\tilde{s}^{m+1}$ are then explicitly calculated, is required such that only linear algebraic equations need to be solved at each time step. This approach maintains high-order accuracy by carefully selecting the approximation of the integration curve $\widetilde{\Gamma}^{m+1}$ to adjust all numerical quantities at the same time level $t_{m+1}$. \emph{Indeed, we emphasize that the selection of $\widetilde{\Gamma}^{m+1}$ is key to the success of this scheme, especially in maintaining good mesh distribution and stability during the evolution. Instead of using the classical extrapolation formulae from the former parameterized functions $\mathbf{X}^{m}$ and $\mathbf{X}^{m-1}$, which might lead to mesh distortion problems \cite{Duan2023}, here we choose the prediction curve $\widetilde{\Gamma}^{m+1}$ (or equivalently $\widetilde{\bX}^{m+1}$) as the solution of the classical first-order BGN scheme \eqref{BGN1semi}. Extensive numerical experiments indicate that expected second-order accuracy in time can be achieved in terms of shape metrics, good mesh quality is maintained during the evolution and mesh distortion can be prevented effectively (see Section 3 and Remark \ref{Example of extrapolation}).} The same idea can be further extended to develop BGN/BDF$k$ schemes for $k=3,4, 5, 6$.

It is worth mentioning that various efforts have been made in the literature to develop high-order temporal schemes for solving geometric flows, each based on different approaches. For example, an implicit Crank-Nicolson-type scheme was designed for forced curve shortening flow by combining with mesh redistribution \cite{Balazovjech-Mikula} or an adaptive moving mesh technique  \cite{Mackenzie-Nolan-Rowlatt-Insall} to maintain a good mesh quality. Both schemes have been shown to converge quadratically in time but require solving a system
of nonlinear equations at each time step.
Based on evolving surface finite element methods and backward differentiation formulae, Kov\'acs, Li and Lubich proposed some high-order numerical schemes for solving mean curvature flow and Willmore flow~\cite{KLL2019,KLL2021,KLL2020}.
Due to the lack of the tangential velocity, these schemes may suffer from mesh clustering and distortion, leading to breakdowns of simulations in some cases. Very recently, by introducing an artificial tangential velocity determined by a harmonic map from a fixed reference surface to the unknown evolving surface, Duan and Li have proposed a new class of parametric finite element methods, including a second-order scheme, with good mesh quality for simulating various types of geometric flows~\cite{Duan2023}.

Among the aforementioned works, our proposed BGN-based high-order schemes inherit most of the advantages from the classical BGN scheme, including:
\begin{itemize}
	\item good mesh quality is maintained during the evolution and no numerical instability occurs;
	
	\item the methods can be easily implemented very efficiently, as only linear algebraic equations need to be solved at each time step;
	
	\item the methods can be extended straightforwardly to a wide range of geometric flows of curves or surfaces, such as area-preserving mean curvature flow, generalized mean curvature flow and Willmore flow;
\end{itemize}
and more importantly,
\begin{itemize}
\item the approach can be  extended to higher-order BDF$k$ schemes which converge at the $k$-th order in time in terms of shape metrics while keeping all the above superiorities.
    \end{itemize}
Compared to the method proposed by \cite{Duan2023} that also employs BDF methods and achieves favorable mesh quality, we  emphasize  that our approach is based on the classical BGN scheme, requiring only minor modifications. Moreover, our innovative strategy of iteratively selecting the prediction curve/polyhedron using low-order BDF$k$ methods effectively combines both the evolving nature of the problem and the desirable mesh properties. This strategy is promising for potential applications in other numerical methods for geometric flows.

\medskip

The rest of this  paper is organized as follows. In Section 2, we provide a brief overview of the classical first-order BGN scheme, using curve shortening flow (CSF) and mean curvature flow (MCF) as examples. In Section 3, we propose high-order BGN/BDF$k$ schemes for solving various types of geometric flows. To demonstrate the accuracy, efficiency, and applicability of our high-order algorithms, we present numerous numerical examples for simulating curve and surface evolution driven by different types of geometric flows in Section 4.
 Finally, we summarize our findings, draw some conclusions based on the results, and discuss potential future research directions in this field in Section 5.

\smallskip

\section{Review of classical BGN scheme}

In this section, we review the classical first-order BGN schemes for CSF and its three-dimensional analogue mean curvature flow (MCF), which were proposed by Barrett, Garcke and N\"urnberg~\cite{BGN07A,BGN07B,BGN08B,BGN20}. To begin with, we rewrite the CSF into the BGN formulation~\eqref{eqn:weak}.

 We introduce the following finite element approximation. Let $\mathbb{I}=[0,1]= \bigcup_{j=1}^N I_j$, $N\ge 3$, be a decomposition of $\mathbb{I}$ into intervals given by the nodes $\rho_j$, $I_j=[\rho_{j-1},\rho_j]$. Let $h=\max\limits_{1\le j\le N}
|\rho_j-\rho_{j-1}|$ be the maximal length of the grid. Define the linear finite element space as
\[
V^h:=\{u\in C(\mathbb{I}): u|_{I_j} \,\,\, \mathrm{is\,\,\,linear,\,\,\,} \forall j=1,2,\ldots,N;\quad u(\rho_0)=u(\rho_N) \}\subseteq H^1(\mathbb{I},\mathbb{R}).
\]
The mass lumped inner product $(\cdot,\cdot)_{\Gamma^h}^h$ over the polygonal $\Gamma^h$, which is an approximation of the inner product $(\cdot,\cdot)_{\Gamma^h}$ by using the composite trapezoidal rule, is defined as
\[
(u,v)_{\Gamma^h}^h:=\frac{1}{2}\sum_{j=1}^N|\bX^h(\rho_j)-\bX^h(\rho_{j-1})|\l[(u\cdot v)(\rho_j^-)+(u\cdot v)(\rho_{j-1}^+) \r],
\]
where $\bX^h$ is a parameterization of $\Gamma^h$, $\bX^h(\rho_j)$ is the vertex of the polygon $\Gamma^h$, and  $u, v$ are two scalar/vector piecewise continuous functions with possible jumps at the nodes $\{\rho_j\}_{j=1}^N$,
and $u(\rho_j^{\pm})=\lim\limits_{\rho\rightarrow \rho_j^{\pm}}u(\rho)$.

Subsequently, the semi-discrete scheme of the formulation \eqref{eqn:weak} is as follows: given initial polygon $\Gamma^h(0)$ with vertices lying on the initial curve $\Gamma(0)$ in a clockwise manner, parametrized by $\bX^h(\cdot,0)\in [V^h]^2$,
find $(\bX^h(\cdot,t),\kappa^h(\cdot,t))\in [V^h]^2\times V^h$ so that
\begin{equation}\label{CSF:Semi-discrete}
	\begin{cases}
			\l(\p_t\mathbf{X}^h\cdot  \mathbf{n}^h,\varphi^h \r)_{\Gamma^h}^h+\l(  \kappa^h,\varphi^h\r)^h_{\Gamma^h}=0,\quad \forall\ \varphi^h\in V^h,\\
			\l(\kappa^h,\mathbf{n}^h\cdot \bm{\omega}^h\r)^h_{\Gamma^h}-\l(\p_s \mathbf{X}^h,\p_s\bm{\omega}^h \r)_{\Gamma^h}=0,\quad \forall\ \bm{\omega}^h\in  [V^h]^2,
		\end{cases}
\end{equation}
where we always integrate over the current curve $\Gamma^h$ described by $\mathbf{X}^h$, the outward unit normal $\bn^h$ is a piecewise constant vector given by
\[\bn^h|_{I_j}=-\frac{\mathbf{h}_j^\perp}{|\mathbf{h}_j|}, \quad \mathbf{h}_j=\bX^h(\rho_j,t)-\bX^h(\rho_{j-1},t),\quad j=1,\ldots, N,\]
with  $\cdot^\perp$ denoting clockwise rotation by $\frac{\pi}{2}$, and
the partial derivative $\p_s$ is defined piecewisely over each side of the polygon
$\p_s f|_{I_j}=\frac{\p_\rho f}{|\p_\rho \mathbf{X}^h|}|_{I_j}=\frac{(\rho_j-\rho_{j-1})\p_\rho f|_{I_j}}{|\mathbf{h}_j|}$. It was shown that the scheme \eqref{CSF:Semi-discrete} will always equidistribute the vertices along $\Gamma^h$ for $t>0$ if they are not locally parallel (see Remark 2.4 in \cite{BGN07A}).

For a full discretization, we fix $\tau>0$ as a uniform time step size for simplicity, and let $\bX^m\in [V^h]^2$ and $\Gamma^m$ be the approximations of $\bX(\cdot,t_m)$ and $\Gamma(t_m)$, respectively, for $m=0,1,2,\ldots$, where $t_m:=m\tau$. We define $\mathbf{h}_j^m:=\bX^m(\rho_j)-\bX^m(\rho_{j-1})$ and assume $|\mathbf{h}_j^m|>0$ for $j=1,\ldots,N$, $\forall\ m>0$. The discrete unit normal vector $\bn^m$, the discrete inner product $(\cdot,\cdot)^h_{\Gamma^m}$ and the discrete operator $\p_s$ are defined similarly as in the semi-discrete case.
Barrett, Garcke and N\"urnberg used a formal first-order approximation \cite{BGN07A,BGN07B} to replace the velocity $\p_t \bX$, $\kappa$ and $\p_s\bX$ by
\begin{equation*}
\begin{split}
     \p_t \bX(\cdot, t_m)&= \frac{\mathbf{X}(\cdot,t_{m+1})-\mathbf{X}(\cdot, t_m)}{\tau}+\mathcal{O}(\tau), \\
	\kappa(\cdot,t_m)&=\kappa(\cdot,t_{m+1})+\mathcal{O}(\tau),  \\
	\p_s\bX(\cdot,t_m)&= \p_s \mathbf{X}(\cdot, t_{m+1})+\mathcal{O}(\tau),
\end{split}
\end{equation*}
and the fully discrete semi-implicit BGN scheme reads as:

(\textbf{BGN1, the classical first-order BGN scheme for CSF}): For $m\ge 0$, find $\mathbf{X}^{m+1}\in [V^h]^2$ and $\kappa^{m+1}\in V^h$ such that
\begin{equation}\label{CSF:BGN1}
		\begin{cases}
			\l(\frac{\mathbf{X}^{m+1}-\mathbf{X}^m}{\tau},\varphi^h \mathbf{n}^m \r)^h_{\Gamma^m}+\l(  \kappa^{m+1},\varphi^h \r)_{\Gamma^m}^h=0,\quad \forall\ \varphi^h\in V^h,\\	\l(\kappa^{m+1},\mathbf{n}^m\cdot \bm{\omega}^h\r)_{\Gamma^m}^h-\l(\p_s \mathbf{X}^{m+1},\p_s\bm{\omega}^h\r)_{\Gamma^m}=0,\quad \forall\ \bm{\omega}^h\in  [V^h]^2.
		\end{cases}
	\end{equation}
The well-posedness and energy stability were established under some mild conditions~\cite{BGN20}. In practice, numerous numerical results show that the classical  BGN scheme \eqref{CSF:BGN1} converges quadratically in space~\cite{BGN07B} and linearly in time~\cite{Jiang23B}.

The above idea has been successfully extended to the mean curvature
flow (MCF) in $\mathbb{R}^3$  \cite{BGN08B}. The governing equation of MCF is given by
\begin{equation}\label{MCF}
		\cV = -\mathcal{H}\, \bn,
	\end{equation}
	where $\mathcal{H}$ is the mean curvature of the hypersurface. Following the lines in \cite{BGN08B,BGN20},  \eqref{MCF} can be
reformulated as
\begin{equation}
\begin{split}
	\p_t \mathbf{X}\cdot \mathbf{n} &=-\mathcal{H},\\
  	\mathcal{H} \mathbf{n}&=-\Delta_\Gamma \mathrm{Id},
\end{split}
\end{equation}
where $\Delta_\Gamma$ is the surface Laplacian (i.e., Laplace-Beltrami operator) and $\mathrm{Id}$ is the identity map on $\Gamma$.  For the finite element approximation, suppose we have a  polyhedra surface $\Gamma^m$ approximating  the closed surface $\Gamma(t_m)$, which is a union of non-degenerate triangles with no hanging vertices
\[
 \Gamma^m:=\bigcup^J_{j=1} \overline{\sigma^m_j},
 \]
where $\{\sigma^m_j\}_{j=1}^J$ is a family of  mutually disjoint open triangles. Set $h=\max\limits_{1\le j\le J}\mathrm{diam}(\sigma_j^m)$.
Denote $W_m^h:=W^h(\Gamma^m)$ by the space of scalar continuous piecewise linear functions on $\Gamma^m$, i.e.,
\[
W_m^h=W^h(\Gamma^m):=\l\{u\in C(\Gamma^m, \mathbb{R}):\ u|_{\sigma^m_j}\ \text{is linear},\ \forall j=1,\ldots,J \r\}\subseteq H^1(\Gamma^m, \mathbb{R}).
\]
Similarly, for piecewise continuous scalar or vector functions $u, v\in L^2(\Gamma^m, \mathbb{R}^3)$ with possible jumps at the edges of  $\sigma^m_j$, the $L^2$ inner product $( \cdot, \cdot)_{\Gamma^m}$ over the polyhedral surface $\Gamma^m$
\[\l( u, v\r)_{\Gamma^m}=\int_{\Gamma^m} u\cdot v \, \mathrm{d} A,\]
can be approximated by the mass lumped inner product
\begin{align*}
\l(u,v\r)_{\Gamma^m}^h:=\frac{1}{3}\sum_{j=1}^J|\sigma^m_j|
\sum_{k=1}^3(u\cdot v)\l((\mathbf{q}_{j_k}^m)^{-} \r),
\end{align*}
where $\{\mathbf{q}_{j_1},\mathbf{q}_{j_2},\mathbf{q}_{j_3}\}$ are the vertices of the triangle $\sigma^m_j$, $|\sigma^m_j|$ is the area of $\sigma^m_j$, and
\[u\l((\mathbf{q}_{j_k}^m)^{-} \r)=\lim\limits_{\sigma^m_j\ni \mathbf{x}\rightarrow \mathbf{q}_{j_k}^m}u(\mathbf{x}).\]

Similarly, the classical first-order BGN scheme for MCF was proposed as~\cite{BGN08B}:

(\textbf{BGN1, the classical BGN, first-order scheme for MCF}): Given $\Gamma^0$ and the identity function $\bX^0\in W_0^h$ on $\Gamma^0$, for $m\ge 0$, find $\mathbf{X}^{m+1}\in [W_m^h]^3$ and $\mathcal{H}^{m+1}\in W_m^h$ such that
\begin{equation}\label{MCF:BGN1}
		\begin{cases}
			\l(\frac{\mathbf{X}^{m+1}-\mathbf{X}^m}{\tau},\varphi^h \mathbf{n}^m \r)^h_{\Gamma^m}+\l(  \mathcal{H}^{m+1},\varphi^h \r)_{\Gamma^m}^h=0,\quad \forall\ \varphi^h\in W_m^h,\\	\l(\mathcal{H}^{m+1},\mathbf{n}^m\cdot \bm{\omega}^h\r)_{\Gamma^m}^h-\l(\nabla_{\Gamma^m}\mathbf{X}^{m+1},\nabla_{\Gamma^m} \bm{\omega}^h\r)_{\Gamma^m}=0,\quad \forall\ \bm{\omega}^h\in  [W_m^h]^3,
		\end{cases}
	\end{equation}
where the outward unit normal vector $\mathbf{n}^m$ of $\Gamma^m$ is defined piecewisely over triangle $\{\sigma_{j}^m\}$ as
\[
\l.\mathbf{n}^m\r|_{\sigma_j^m}=\mathbf{n}_j^m:=\frac{(\mathbf{q}^m_{j_2}-\mathbf{q}^m_{j_1} )\times (\mathbf{q}^m_{j_3}-\mathbf{q}^m_{j_1}) }{|(\mathbf{q}^m_{j_2}-\mathbf{q}^m_{j_1} )\times (\mathbf{q}^m_{j_3}-\mathbf{q}^m_{j_1})|},
\]
and $\mathbf{X}^m(\cdot)$ is the identity function on $[W_m^h]^3$. The surface gradient $\nabla_{\Gamma}$ over polyhedra $\Gamma$ is defined for $f\in W^h(\Gamma)$ piecewisely on triangles $\{\sigma_j\}$ with vertices $\{\mathbf{q}_{j_1},\mathbf{q}_{j_2},\mathbf{q}_{j_3} \}$ as
\[
(\nabla_{\Gamma}f)|_{\sigma_j}:=f(\mathbf{q}_{j_1})\frac{(\mathbf{q}_{j_3}
-\mathbf{q}_{j_2})\times \mathbf{n}_j}{|\sigma_j|}+f(\mathbf{q}_{j_2})\frac{(\mathbf{q}_{j_1}-\mathbf{q}_{j_3})
\times \mathbf{n}_j}{|\sigma_j|}+f(\mathbf{q}_{j_3})\frac{(\mathbf{q}_{j_2}-
\mathbf{q}_{j_1})\times \mathbf{n}_j}{|\sigma_j|},
\]
The well-posedness and energy stability were also  established under some mild conditions \cite{BGN08B}, and quadratic convergence rate  in space was reported.

Throughout the paper, the first-order BGN scheme \eqref{CSF:BGN1} or \eqref{MCF:BGN1} is referred to as BGN1 scheme.

 \smallskip

\section{High-order in time, BGN-based algorithms}
In this section, we propose  high-order temporal  schemes based on the backward differentiation formulae (BDF). For simplicity of notation, here we only present the schemes for the CSF and similar schemes can be proposed for MCF in $\mathbb{R}^3$. Specifically, we approximate the term  $\p_t\bX$ based on the following Taylor expansions~\cite{Leveque}:
\begin{align*}
	\p_t \bX(\cdot, t_{m+1})=
	\begin{cases}
		&\frac{ \frac{3}{2}\bX(\cdot,t_{m+1})-2\bX(\cdot,t_{m})+\frac{1}{2}\bX(\cdot,t_{m-1})}{\tau}+\mathcal{O}(\tau^2),\\
		&\frac{ \frac{11}{6}\bX(\cdot,t_{m+1})-3\bX(\cdot,t_{m})+\frac{3}{2}\bX(\cdot,t_{m-1})-\frac{1}{3}\bX(\cdot,t_{m-2})}{\tau}+\mathcal{O}(\tau^3),\\
		&\frac{ \frac{25}{12}\bX(\cdot,t_{m+1})-4\bX(\cdot,t_{m})+3\bX(\cdot,t_{m-1})-\frac{4}{3}\bX(\cdot,t_{m-2})+\frac{1}{4}\bX(\cdot,t_{m-3})}{\tau}+\mathcal{O}(\tau^4).
	\end{cases}
\end{align*}
Thus the velocity is  approximated  with an error of $\mathcal{O}(\tau^k)$, $2\le k\le 4$ at the time level $t_{m+1}$.
By taking the mass lumped inner product over a suitable predictor  $\widetilde{\Gamma}^{m+1}$, which is an approximation of $\Gamma(t_{m+1})$, we obtain the following high-order schemes (denoted as BGN/BDF$k$ schemes) for $2\le k\le 4$:

\smallskip

(\textbf{BGN/BDF$k$, High-order schemes for CSF}):~ For $k=2,3,4$, $m\ge k-1$, find $\mathbf{X}^{m+1}\in [V^h]^2$ and $\kappa^{m+1}\in V^h$  such that
\begin{equation}\label{CSF:BDFk}
		\begin{cases}
	\l(\frac{a \mathbf{X}^{m+1}-\mathbf{\widehat{X}}^{m}}{\tau},\varphi^h \mathbf{\widetilde{n}}^{m+1} \r)^h_{\widetilde{\Gamma}^{m+1}}+\l(  \kappa^{m+1},\varphi^h \r)_{\widetilde{\Gamma}^{m+1}}^h=0,\quad \forall\ \varphi^h\in V^h,\\
	\vspace{-3mm}\\		\l(\kappa^{m+1},\mathbf{\widetilde{n}}^{m+1}\cdot \bm{\omega}^h\r)_{\widetilde{\Gamma}^{m+1}}^h-\l(\p_s \mathbf{X}^{m+1},\p_s\bm{\omega}^h\r)_{\widetilde{\Gamma}^{m+1}}=0,\quad \forall\ \bm{\omega}^h\in  [V^h]^2,
		\end{cases}
	\end{equation}
where $a$, $\mathbf{\widehat{X}}^{m}$ are defined as
\begin{align}
	&\label{CSF:BDF2} \text{BDF2}: a = \frac{3}{2},\ \ \qquad \mathbf{\widehat{X}}^{m}=2\mathbf{X}^{m}-\frac{1}{2}\mathbf{X}^{m-1};\\
	&\label{CSF:BDF3}\text{BDF3}: a = \frac{11}{6},\qquad \mathbf{\widehat{X}}^{m}=3\mathbf{X}^{m}-\frac{3}{2}\mathbf{X}^{m-1}+\frac{1}{3}\mathbf{X}^{m-2};\\
	&\label{CSF:BDF4}\text{BDF4}: a = \frac{25}{12},\qquad \mathbf{\widehat{X}}^{m}=4\mathbf{X}^{m}-3\mathbf{X}^{m-1}+\frac{4}{3}\mathbf{X}^{m-2}-\frac{1}{4}\mathbf{X}^{m-3},\end{align}
where  $\widetilde{\Gamma}^{m+1}$, described by $\widetilde{\bX}^{m+1}\in [V^h]^2$, is a suitable approximation of
$\Gamma(t_{m+1})$,  $\widetilde{\mathbf{n}}^{m+1}:=-\l(\frac{\p_\rho \widetilde{\bX}^{m+1}  }{|\p_\rho \widetilde{\bX}^{m+1} |}\r)^{\perp}$ is the normal vector, and the derivative $\p_s$ is defined with respect to the arc-length of $\widetilde{\Gamma}^{m+1}$.
	
\smallskip

Before introducing the specific choice of the  approximation $\widetilde{\Gamma}^{m+1}$,  we first establish  the fundamental properties of above BGN/BDF$k$ schemes \eqref{CSF:BDFk}. The first crucial  property is the well-posedness of \eqref{CSF:BDFk} under some mild conditions.
	
\begin{theorem}[Well-posedness]\label{Well-posedness}
 For $k=2,3,4$, $m\ge k-1$,  assume that the polygon $\widetilde{\Gamma}^{m+1}$ in the BGN/BDF$k$ schemes \eqref{CSF:BDFk} satisfies the following two conditions:

  \noindent(i) At least two vectors in $\{\widetilde{\mathbf{h}}^{m+1}_j\}_{j=1}^{N}$ are not parallel, i.e.,
  \[\mathrm{dim}\l( \mathrm{Span}\l\{\widetilde{\mathbf{h}}^{m+1}_j \r\}_{j=1}^{N}\r)=2;\]

	\noindent	(ii) No vertices degenerate  on $\widetilde{\Gamma}^{m+1}$, i.e., \[\min_{1\le j\le N}| \widetilde{\mathbf{h}}^{m+1}_j|>0.\]
Then the above BGN/BDF$k$ schemes \eqref{CSF:BDFk} are well-posed, i.e., there exists a unique solution $(\mathbf{X}^{m+1},\kappa^{m+1})\in [V^h]^2\times V^h$ of \eqref{CSF:BDFk}.
\end{theorem}

\begin{proof}
Thanks to the linearity of the scheme \eqref{CSF:BDFk}, it suffices to  prove that the following algebraic system for
$(\bX,\kappa)\in [V^h]^2\times V^h$  has only zero solution,
	\begin{equation*}
		\begin{cases}
			\l(\frac{a\mathbf{X}}{\tau},\varphi^h \widetilde{\mathbf{n}}^{m+1} \r)^h_{\widetilde{\Gamma}^{m+1}}+\l(\kappa ,\varphi^h \r)_{\widetilde{\Gamma}^{m+1}}^h=0,\quad \forall\ \varphi^h\in V^h,\\
			\vspace{-3mm}\\
   \l(\kappa,\widetilde{\mathbf{n}}^{m+1}\cdot \bm{\omega}^h\r)_{\widetilde{\Gamma}^{m+1}}^h-\l(\p_s \mathbf{X},\p_s\bm{\omega}^h\r)_{\widetilde{\Gamma}^{m+1}}=0,\quad \forall\ \bm{\omega}^h\in  [V^h]^2.
		\end{cases}
	\end{equation*}
	Taking $\varphi^h=\kappa$ and $\bm{\omega}^h=\bX$, noticing that $a>0$, we arrive at
	\begin{equation}
		\bX\equiv \bX^c,\qquad \kappa \equiv \kappa^c.
	\end{equation}
	Then the standard argument in \cite[Theorem 2.1]{BGN07A} yields $\bX^c=0$ and $\kappa^c=0$  by  the assumption on $\widetilde{\Gamma}^{m+1}$.
	\end{proof}

\smallskip

It remains to determine $\widetilde{\Gamma}^{m+1}$, or equivalently $\widetilde{\bX}^{m+1}$. Indeed, the formulation of this predictor plays a crucial role in the success of BGN/BDF$k$ schemes. A natural approach is to apply standard extrapolation formulae, such as
  \begin{numcases}{\widetilde{\bX}^{m+1}=}
2\bX^{m}-\bX^{m-1},\quad  & \text{BDF2}, \label{exbdf2}\\
3\bX^{m}-3\bX^{m-1}+\bX^{m-2},\quad  & \text{BDF3}, \label{exbdf3}\\
	4\bX^{m}-6\bX^{m-1}+4\bX^{m-2}-\bX^{m-3},\quad  & \text{BDF4}. \label{exbdf4}
\end{numcases}
Unfortunately, as discussed in \cite{Duan2023}, naively extrapolating the curves as functions may result in the instability of high-order schemes (see also Remark \ref{Example of extrapolation}). In this paper,
we utilize the solution of the lower-order BGN/BDF$k$ scheme to predict  the discrete polygon $\widetilde{\Gamma}^{m+1}$, or equivalently $\widetilde{\bX}^{m+1}$, in the BGN/BDF$k$ scheme. In particular, for $k=1$, BGN/BDF1 scheme represents the classical BGN1 scheme, which is used to predicate $\widetilde{\Gamma}^{m+1}$ in BGN/BDF2 scheme.

To start BGN/BDF$k$ scheme, it is necessary to prepare the initial data $\bX^0,\ldots, \bX^{k-1}$, which are supposed to be approximations of $\bX(\cdot, 0),\ldots, \bX(\cdot, t_{k-1})$ with error at the $k$-th order $\mathcal{O}(\tau^{k})$. This can be  accomplished  by utilizing BGN1 scheme \eqref{CSF:BGN1} with a finer time step. Specifically, to obtain an approximation of $\bX(\cdot, t_1)$ with error at $\mathcal{O}(\tau^k)$, where $k=2,3,4$, it is  sufficient to implement BGN1 scheme with a time step size $\widetilde{\tau}\sim \tau^{k-1}$  by $\tau/\widetilde{\tau}$ steps. Taking into account the truncation error of BGN1 scheme \eqref{CSF:BGN1}, the accumulated temporal error at $t_1=\tau$ is given by
	\[
	\widetilde{\tau}^2* \tau/\widetilde{\tau}\sim \tau^k,\quad k=2,3,4.
	\] 		
\smallskip

Now we are ready to present our BGN/BDF$k$ schemes. Throughout all the algorithms, we always use $\Gamma^m$ and $\widetilde{\Gamma}^m$ to represent the polygon described by the vector functions $\bX^m, \widetilde{\bX}^m\in [V^h]^2$, respectively.

\begin{algorithm}
	BGN/BDF2 algorithm
	\label{BGN/BDF2 algorithm}
\begin{algorithmic}[1]
 \REQUIRE An initial curve $\Gamma(0)$ approximated by a polygon $\Gamma^0$ with $N$ vertices, described by $\bX^0\in [V^h]^2$ , time step $\tau$ and terminate time $T$ satisfying $T/\tau\in\mathbb{N}$.
 \STATE Calculate $\bX^{1}$ by using BGN1 scheme \eqref{CSF:BGN1} with $\Gamma^0$ and $\tau$. Set $m=1$.
 \WHILE{$m < T/\tau$,}
 \STATE {\small
 \begin{itemize}
\item Compute  $\widetilde{\bX}^{m+1}$ by using BGN1 scheme \eqref{CSF:BGN1}
with $\Gamma^m$ and $\tau$.
 \item Compute $\bX^{m+1}$ by using the BGN/BDF2 scheme \eqref{CSF:BDFk}-\eqref{CSF:BDF2} with
 $\bX^{m-1}$, $\bX^{m}$, $\widetilde{\Gamma}^{m+1}$ and $\tau$.
\end{itemize}
}
 \STATE $m = m + 1$;
 \ENDWHILE
 \end{algorithmic}
\end{algorithm}

\smallskip

\begin{algorithm}
	BGN/BDF3 algorithm
	\label{BGN/BDF3 algorithm}
\begin{algorithmic}[1]
 \REQUIRE An initial curve $\Gamma(0)$ approximated by a  polygon $\Gamma^0$ (parametrized as $\bX^0\in[V^h]^2$) with $N$ vertices, terminate time $T$, time step $\tau$ and a finer time step $\widetilde{\tau}\sim \tau^2$.
 \STATE Calculate $\bX^{1}$ by using BGN1 scheme \eqref{CSF:BGN1} with $\Gamma^0$ and $\widetilde{\tau}$ for $\tau/\widetilde{\tau}$ steps.
 \STATE Calculate $\bX^{2}$ by using BGN/BDF2 algorithm (Algorithm \ref{BGN/BDF2 algorithm}) with  $\bX^{0}$, $\bX^{1}$ and
 $\tau$. Set $m=2$.
 \WHILE{$m < T/\tau$,}
 \STATE {\small
 \begin{itemize}
\item Compute $\widetilde{\bX}^{m+1}$ by using BGN/BDF2  algorithm (Algorithm \ref{BGN/BDF2 algorithm}) with
$\bX^{m-1}$, $\bX^{m}$ and $\tau$.
 \item Compute $\bX^{m+1}$ by using the BGN/BDF3 scheme \eqref{CSF:BDFk} and \eqref{CSF:BDF3} with
 $\bX^{m-2}$, $\bX^{m-1}$, $\bX^{m}$, $\widetilde{\Gamma}^{m+1}$  and $\tau$.
\end{itemize}
}
 \STATE $m = m + 1$;
 \ENDWHILE
 \end{algorithmic}
\end{algorithm}

\begin{algorithm}
	BGN/BDF4 algorithm
	\label{BGN/BDF4 algorithm}
\begin{algorithmic}[1]
 \REQUIRE An initial curve $\Gamma(0)$ approximated by a  polygon $\Gamma^0$ (parametrized as $\bX^0\in[V^h]^2$) with $N$ vertices, terminate time $T$, time step $\tau$ and a finer time step $\widetilde{\tau}\sim \tau^3$.
 \STATE Calculate $\bX^{1}$ by using BGN1 scheme \eqref{CSF:BGN1} with $\Gamma^0$ and $\widetilde{\tau}$ for $\tau/\widetilde{\tau}$ steps.
 \STATE Calculate $\bX^{2}$ by using BGN1 scheme \eqref{CSF:BGN1} with $\Gamma^1$ and $\widetilde{\tau}$ for another $\tau/\widetilde{\tau}$ steps.
 \STATE Calculate $\bX^{3}$ by using BGN/BDF3 algorithm (Algorithm \ref{BGN/BDF3 algorithm}) with  $\bX^{0}$, $\bX^{1}$,
 $\bX^{2}$ and  $\tau$. Set $m=3$.
 \WHILE{$m < T/\tau$,}
 \STATE {\small
 \begin{itemize}
\item Compute $\widetilde{\bX}^{m+1}$ by BGN/BDF3  algorithm (Algorithm \ref{BGN/BDF3 algorithm}) with
$\bX^{m-2}$, $\bX^{m-1}$, $\bX^{m}$ and $\tau$.
 \item Compute $\bX^{m+1}$ by the BGN/BDF4 scheme \eqref{CSF:BDFk} and \eqref{CSF:BDF4} with $\bX^{m-3}$, $\bX^{m-2}$, $\bX^{m-1}$, $\bX^{m}$, $\widetilde{\Gamma}^{m+1}$  and $\tau$.
\end{itemize}
}
 \STATE $m = m + 1$;
 \ENDWHILE
 \end{algorithmic}
\end{algorithm}

%
%
%

\medskip

\begin{remark}
	Similar approach can be utilized to construct  BGN/BDF5 and BGN/BDF6 algorithms, but for the sake of brevity, we omit the details here.
\end{remark}
\medskip

It is worth noting that the classical first-order BGN scheme exhibits favorable properties in terms of mesh distribution.
As observed in \cite{BGN07B,BGN08B}, the node points are moved tangentially, resulting in eventual equidistribution in practice. Furthermore, it has been demonstrated in previous studies \cite{BGN20,Zhao-Jiang-Bao2021} that the mesh will eventually become evenly distributed if the solution has an equilibrium state. In other words, the mesh becomes asymptotically equidistributed when the solution has a non-degenerate equilibrium, such as in the case of area-preserving curve shortening flow or surface diffusion.
Additionally, recent discoveries have highlighted the significant role of BGN1 scheme in improving mesh quality when implementing a second-order BGN-type Crank-Nicolson leap-frog scheme occasionally \cite{Jiang23B}.
It is worth emphasizing that BGN1 scheme is utilized at every step of the mentioned BDF$k$ algorithms, offering great promise in ensuring favorable mesh distribution(cf. Section 4) and preventing mesh distortion or numerically induced self-intersections.

\smallskip

\section{Numerical results}

In this section, we provide numerous numerical examples to demonstrate the high-order convergence and advantageous properties in terms of mesh distribution of our BGN/BDF$k$ algorithms for various types of geometric flows.

\subsection{Curve evolution}

 In this subsection, we mainly  focus on the  following types of geometric flows of curves in the plane:
\begin{itemize}
	\item Curve shortening flow (CSF), which is the $L^2$-gradient flow of the length functional $E(\Gamma)=\int_\Gamma \mathrm{d} s$;
	
	\item Area-preserving curve shortening flow (AP-CSF), which is the  $L^2$-gradient flow of length functional with the constraint of a fixed enclosed area. This is a nonlocal second-order geometric flow with the governing equation:
	\begin{equation}\label{AP-CSF}
		\cV = (-\kappa+\l<\kappa\r>)\, \bn,
	\end{equation}
	where $\l<\kappa \r>:=\int_{\Gamma(t)}\kappa\, \d s/\int_{\Gamma(t)}\, \d s$ is the average curvature;
	
	\item Generalized mean curvature flow (G-MCF), which is given by
	\begin{equation}\label{G-MCF}
		\cV = -\beta \kappa^{\alpha}\,\bn,
	\end{equation}
		where $\alpha$ and $\beta$ are two real numbers satisfying $\alpha\beta>0$. For $0<\alpha \neq 1$ and $\beta=1$, it is called as the powers of mean curvature flow; for $\alpha=-1$ and $\beta=-1$, it is called as the inverse mean curvature flow;
		
	\item Willmore flow (WF), which is the $L^2$-gradient flow of Willmore energy functional
$W(\Gamma)=\frac{1}{2}\int_{\Gamma}\kappa^2\ \d s$. This is a fourth-order geometric flow with the governing equation:
	\begin{equation}\label{WF}
		\cV = \l(\p_{ss}\kappa-\frac{1}{2}\kappa^3\r)\, \bn.
	\end{equation}
\end{itemize}
Based on the wide applicability of BGN-type methods \cite{BGN07A,BGN07B,BGN08A,BGN08B,BGN20}, we can easily extend our BGN/BDF$k$  algorithms to the aforementioned geometric flows with minor adjustments.

To test the convergence rate of our proposed BGN/BDF$k$ schemes, it is advisable to employ shape metrics for measuring the numerical errors of the BGN-type schemes, as they allow an intrinsic tangential movement in order to facilitate a more uniform distribution of mesh points along the evolving curve or surface (see \cite[Section 3]{Jiang23B}). In the following examples, we utilize the manifold distance to quantify the difference between two curves. Specifically, the manifold distance between two curves $\Gamma_1$ and $\Gamma_2$ is defined as~\cite{Zhao-Jiang-Bao2021,Jiang23B}:
 \begin{equation}\label{Manifold distance}
	\mathrm{M}\l(\Gamma_1,\Gamma_2\r)
	: = |\Omega_1 \Delta \Omega_2 | =|\Omega_1 |+|\Omega_2 |-2 |\Omega_1\cap \Omega_2 |,\quad \Omega_1\Delta \Omega_2=(\Omega_1\setminus\Omega_2)\cup (\Omega_2\setminus\Omega_1),
\end{equation}
where $\Omega_1$ and $\Omega_2$ represent the regions enclosed by the two curves $\Gamma_1$ and $\Gamma_2$, respectively, and
$|\Omega|$ denotes the area of $\Omega$. It has been proved that the manifold distance satisfies the properties of symmetry, positivity and triangle inequality (see Proposition 5.1 in \cite{Zhao-Jiang-Bao2021}). Thus it can be seen as a type of shape metric. Under some suitable assumptions, it is related to the $L^1$ norm of a distance function. Specifically, assume that $\Gamma_1$ is a $C^2$ curve and $\Gamma_{\delta}:=\{x\in \R^{2}|\, |\mathrm{d}(x,\Gamma_1)|<\delta\}$ is its tabular neighborhood \cite{DDE2005}, where $\mathrm{d}(x,\Gamma_1)$ is the signed distance function. If $\Gamma_2\subset \Gamma_{\delta}$ and the projection map from $\Gamma_2$ to $\Gamma_1$ is one-to-one, then the manifold distance is indeed the $L^1$ norm of the distance function: \[
\mathrm{M}(\Gamma_1,\Gamma_2)=\int_{\Gamma_1} \l|\,\mathrm{d}(\cdot,\Gamma_2)\,\r|\, \mathrm{d}s.
\]
Very recently, Bai and Li \cite{Bai2023} have given a proof of convergence for Dziuk's parametric finite element method with finite elements of degree $k\ge 3$ for mean curvature flow. Their convergence result holds in the $L^2$ norm, which measures $L^2$ norm of the distance between two curves or surfaces. We also note that since numerical solutions are here represented as polygonal curves/polyhedron surfaces, it is very easy
to calculate the area/volume of the symmetric difference region, i.e., the manifold distance.
In practice, the computation of manifold distance can be conveniently performed using the MATLAB packages  \texttt{polyshape} and \texttt{alphaShape}.

To test the temporal errors, we fix the number of nodes $N$ large enough so that the spatial error can be neglected, and we vary the time step. The numerical error and the corresponding convergence order at time $T$ under the manifold distance are then computed as follows:
 \begin{equation}
 \label{eqn:errordef1}
\cE_{\tau}(T)=\cE_{M}(T)= \mathrm{M}(\bX^{T/\tau}, \bX_{\mathrm{ref}}),
\quad \text{Order}=\log\Big(\frac{\cE_{\tau_1}(T)}{\cE_{\tau_2}
(T)} \Big)\Big/ \log\Big(\frac{\tau_1}{\tau_2}\Big),
\end{equation}
where $\bX^{T/\tau}$ represents the computed solutions using a time step of $\tau$  until time $T$, and $\bX_{\mathrm{ref}}$ represents  the reference solutions, either  given by the exact solution or approximated by some numerical solution with a sufficiently fine mesh size and tiny time step when the true solution is not available.

To test the stability properties of our proposed schemes, we investigate the evolution of the following three important quantities with respect to evolving curves: (1) the relative area loss $\Delta A(t)$, (2) the normalized perimeter $L(t)/L(0)$, and (3) the mesh distribution function
$\Psi(t)$, which are defined respectively as follows:
\[
\Delta A(t)|_{t=t_m}=\frac{A^m-A^0}{A^0},\quad \l.\frac{L(t)}{L(0)}\r|_{t=t_m}=\frac{L^m}{L^0},\quad \Psi(t)|_{t=t_m}=\Psi^m:=\frac{\max_j|\mathbf{h}^m_j| }{\min_j|\mathbf{h}^m_j|},
\]
where $A^m$ and $L^m$ represent the enclosed area and the perimeter of the polygon determined by $\bX^m$, respectively.

\smallskip

\begin{example}[Convergence rate of BGN/BDFk scheme for CSF]
We check the convergence rate of the  classical BGN1 scheme \eqref{CSF:BGN1} and BGN/BDF$k$ schemes \eqref{CSF:BDFk} for CSF. The initial curve is chosen as either a unit circle or an ellipse defined by $x^2/4+y^2=1$. The parameters are set as $N=10000$, $T=0.25$.
\end{example}

\smallskip

For the case of unit circle, it is well-known that the CSF admits the exact solution
\[
\bX_{\mathrm{true}}(\rho,t)= \sqrt{1-2t}(\cos(2\pi\rho),\sin(2\pi\rho)),\quad \rho\in \mathbb{I}, \quad t\in [0,0.5).
\]

\begin{figure}[htpb]
\centering
\includegraphics[width=5.1in,height=3.6in]{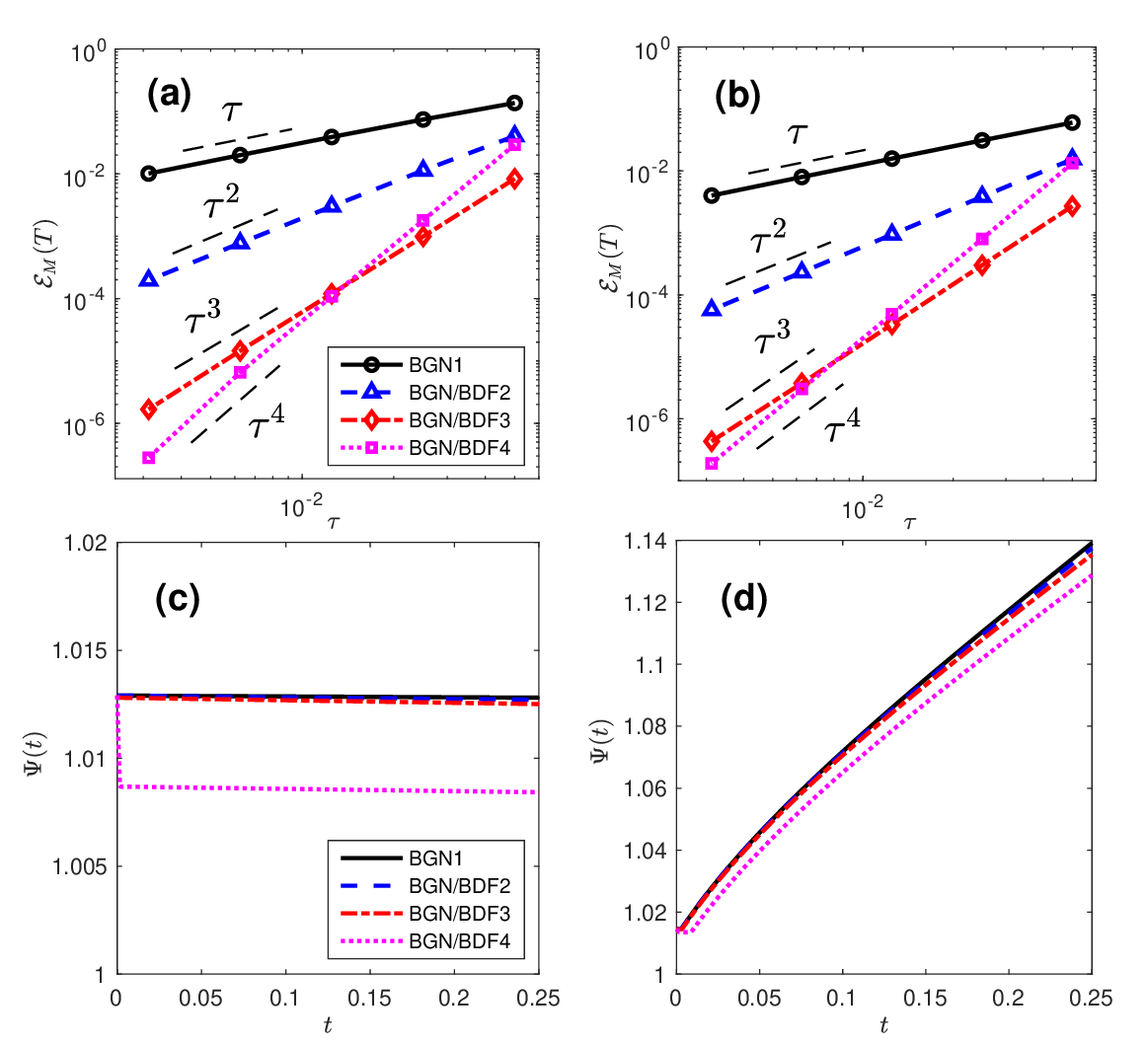}
\setlength{\abovecaptionskip}{2pt}
\caption{Log-log plot of the manifold distance errors via the classical BGN1 scheme and the BGN/BDF$k$ schemes ($2\le k\le 4$) at time $T=0.25$ for solving CSF with two various initial curves: (a) unit circle and (b) ellipse; and  the evolution of the mesh distribution function $\Psi(t)$: (c) unit circle and (d) ellipse, where $N=640$ and $\tau=1/1280$.}
\label{Fig:CSF}
\end{figure}

While for ellipse case, since the true solution is unavailable, we compute the reference solution using BGN/BDF4 algorithm with a fine mesh including $N=10000$ nodes and a tiny time step of $\tau=1/2560$. Figure \ref{Fig:CSF}(a)-(b) present a log-log plot of the numerical errors at time $T=0.25$ for the classical BGN1 scheme \eqref{CSF:BGN1} and BGN/BDF$k$ schemes  \eqref{CSF:BDFk}. It can be clearly observed that the classical BGN1 scheme has only first-order accuracy in time, while the BGN/BDF$k$ scheme achieves $k$th-order accuracy with $2\le k\le 4$ for both the unit circle and ellipse cases. Furthermore,  Figure \ref{Fig:CSF}(c) and (d) show the evolution of the mesh distribution function $\Psi(t)$ of BGN/BDF$k$ algorithms, from which we clearly see that the proposed BGN/BDF$k$ schemes share the same favorable properties with respect to the mesh distribution. Once the initial curve is approximated by a polygon with high mesh quality, the mesh will maintain its quality throughout the evolution. This characteristic significantly enhances the robustness of our BGN/BDF$k$ schemes.

\begin{figure}[htpb]
\centering
\includegraphics[width=5.1in,height=3.6in]{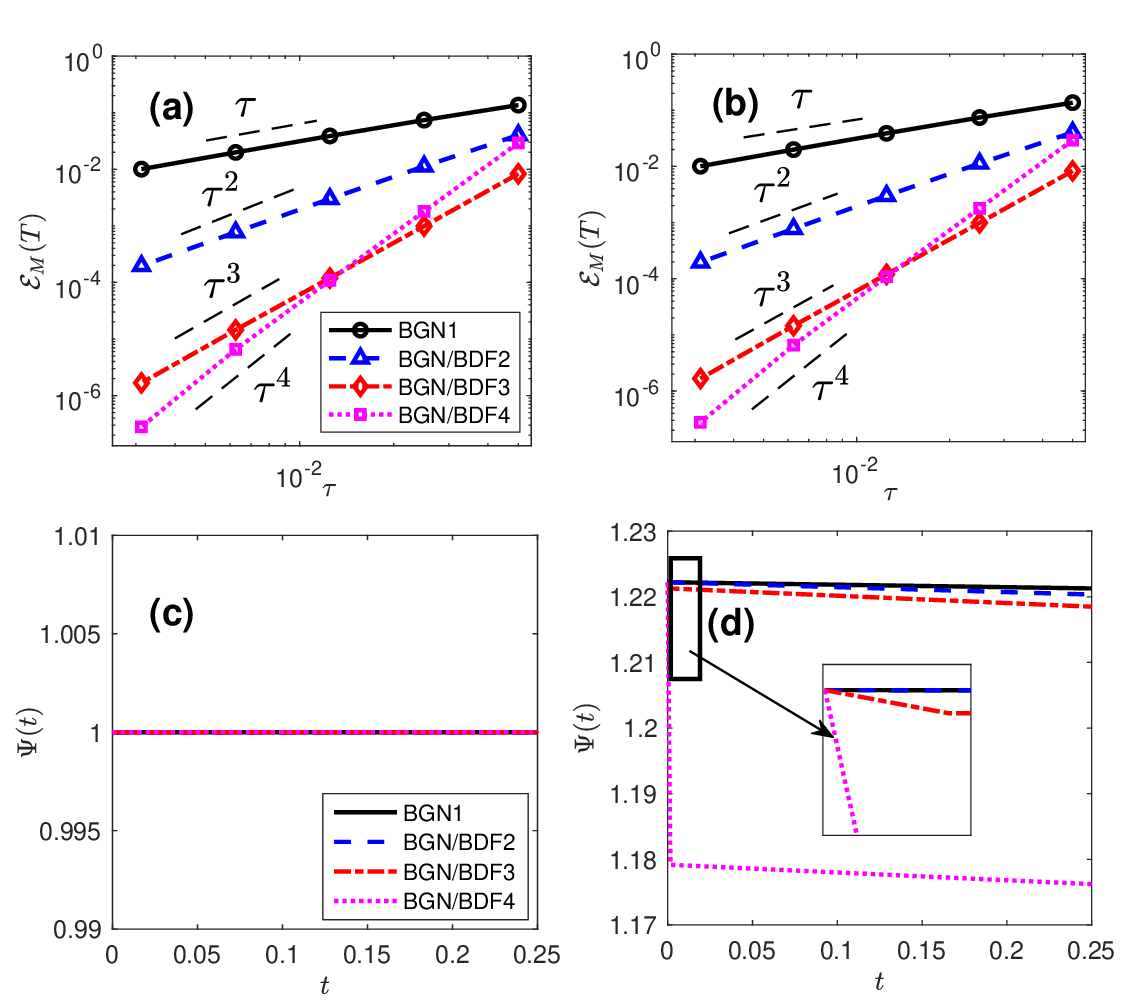}
\setlength{\abovecaptionskip}{2pt}
\caption{Log-log plot of the manifold distance errors at time $T=0.25$ for solving CSF with a unit circle being initial curve: (a) uniform initial  distribution; (b) nonuniform initial distribution; and  the evolution of the mesh distribution function $\Psi(t)$: (c) uniform initial  distribution and (d) nonuniform initial distribution, where $N=640$ and $\tau=1/1280$.}
\label{Fig:CSF_re}
\end{figure}

To further investigate the impact of the initial parametrization on the high-order accuracy and mesh quality of the BGN/BDF$k$ algorithms, we conduct additional experiments using the unit circle as the initial curve. We consider two choices for approximating the $N$-polygon:
\begin{itemize}
	\item Regular polygon: the nodes are set as
	\[
	\mathbf{X}_i=(\cos(2\pi i/N),\sin(2\pi i/N)),\quad i=1,\ldots,N.
	\]
	This ensures an initial mesh ratio of $1$ up to machine precision.
	\item Irregular polygon: the nodes are chosen as:
	\[
		\mathbf{X}_i=(\cos(2\pi i/N+0.1\,\sin(2\pi i/N)),\sin(2\pi i/N+0.1\,\sin(2\pi i/N)),\,\, i=1,\ldots,N.
	\]
This type of distribution can be traced back to \cite[Section 3.1]{BGN07B}. It is clear that the initial mesh ratio deviates from $1$.
\end{itemize}

Figure \ref{Fig:CSF_re}(a)-(b) demonstrate the robust temporal convergence of the BGN/BDF$k$ algorithms of $k$-th order for both initial distributions. Additionally, Figure \ref{Fig:CSF_re}(c) shows that when using a  uniform initial distribution, the mesh ratio remains close to $1$ up to machine precision. On the other hand, Figure \ref{Fig:CSF_re}(d) illustrates that the BGN/BDF$4$ algorithm significantly reduces the mesh ratio at the first time step. This is because we compute the polygon $\Gamma^1$ using the BGN1 scheme for $\tau/\widetilde{\tau}$ steps (cf. Algorithm \ref{BGN/BDF4 algorithm}), and the BGN1 scheme demonstrates favorable properties in terms of mesh distribution for curve evolution \cite{BGN07B,BGN08B,Jiang23B}.

\smallskip

\begin{example}[Extension to AP-CSF]
We extend the BGN/BDF$k$ schemes to AP-CSF and check the convergence rate of BGN1 scheme and BGN/BDF$k$ schemes. The initial curve is chosen as an ellipse. The parameters are chosen as $N=10000$, $T=0.25$ and $1$.
\end{example}

\smallskip

We briefly show how to construct BGN/BDF$k$ schemes for AP-CSF. We first write the  governing equation of AP-CSF into the following coupled equations:
\begin{equation}\label{AP-CSF:Coupled equation}
\begin{split}
	\p_t \mathbf{X}\cdot \mathbf{n} &=-\kappa+\l<\kappa\r>,\\
  	\kappa \mathbf{n}&=-\p_{ss}\mathbf{X}.
\end{split}
\end{equation}
The corresponding BGN1 scheme and BGN/BDF$k$ schemes adjust the first equation in \eqref{CSF:BGN1} and \eqref{CSF:BDFk} to their  nonlocal versions as
\begin{align*}
	\l(\frac{\mathbf{X}^{m+1}-\mathbf{X}^{m}}{\tau},\varphi^h \mathbf{n}^{m} \r)^h_{\Gamma^{m}} &=-\l( \kappa^{m+1}-\l<\kappa^{m+1}\r>_{\Gamma^m} ,\varphi^h \r)_{\Gamma^m}^h,\\
	\l(\frac{a \mathbf{X}^{m+1}-\mathbf{\widehat{X}}^{m}}{\tau},\varphi^h \mathbf{\widetilde{n}}^{m+1} \r)^h_{\widetilde{\Gamma}^{m+1}}&=-\l(  \kappa^{m+1}-\l<\kappa^{m+1}\r>_{\widetilde{\Gamma}^{m+1}},\varphi^h \r)_{\widetilde{\Gamma}^{m+1}}^h,
\end{align*}
respectively, where $\l<\kappa^{m+1}\r>_{\Gamma^{m}}:=\frac{(\kappa^{m+1},1)^h_{\Gamma^{m}}}{(1,1)^h_{\Gamma^{m}} }$, $\l<\kappa^{m+1}\r>_{\widetilde{\Gamma}^{m+1}}:=\frac{(\kappa^{m+1},1)^h_{\widetilde{\Gamma}^{m+1}}}{(1,1)^h_{\widetilde{\Gamma}^{m+1}} }$ and $a,\widehat{\bX}^m$ are defined by \eqref{CSF:BDF2}-\eqref{CSF:BDF4}.

For the test of temporal convergence rate, we take ellipse as an example and similarly compute the reference solution using a fine mesh size of $N=10000$ and a time step of $\tau=1/1280$. Figure \ref{Fig:AP-CSF}(a)-(b) show the comparison of temporal convergence rate of BGN1 scheme and BGN/BDF$k$ schemes at  different times $T=0.25,1$.  It can be clearly observed that the numerical error of  BGN/BDF$k$ schemes converge in $k$-th order, while the classical BGN1 scheme converges only linearly. The evolution of mesh distribution function is depicted in Figure  \ref{Fig:AP-CSF}(c) for different schemes with the same computational parameters $N=640$ and
$\tau=1/1280$. It is worth noting that for long-time simulation, the mesh distribution function $\Psi(t)$ approaches $1$, indicating the long-time asymptotic mesh equidistribution, for all
kinds of algorithms. Furthermore, we also observe that the high-order algorithms will achieve equidistributed mesh faster and the the mesh distribution function is decreasing with respect to the order $k$, showing the advantage of using high-order BGN/BDF$k$ schemes in achieving a more evenly distributed mesh.

\begin{figure}[htpb]
\centering
\includegraphics[width=5.1in,height=1.8in]{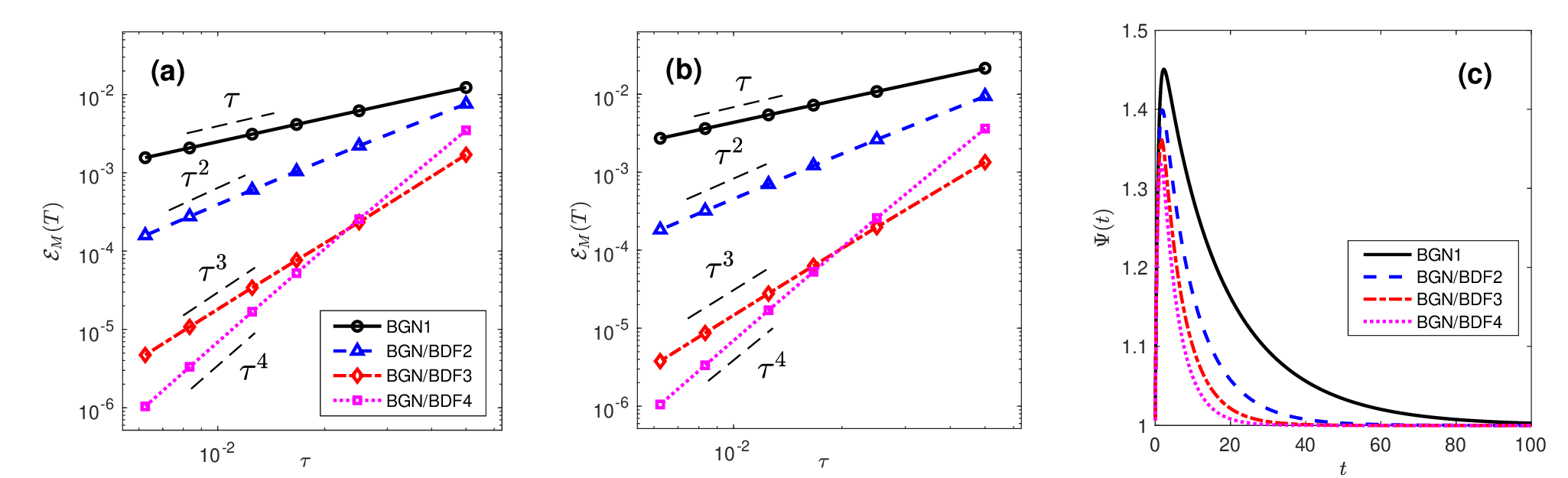}
\setlength{\abovecaptionskip}{2pt}
\caption{Log-log plot of the manifold distance errors for solving AP-CSF with ellipse being its initial shape at two different times: (a) $T=0.25$,  (b) $T=1$. (c) The corresponding evolution of the mesh distributional function $\Psi(t)$, where $N=640$ and $\tau=1/1280$.}
\label{Fig:AP-CSF}
\end{figure}

\smallskip

\begin{remark}\label{Example of extrapolation}
	 In Section 3, we highlight the limitations of standard extrapolation approximations for the prediction polygon $\widetilde{\Gamma}^{m+1}$. To illustrate this point, we use BGN/BDF3 algorithm to solve AP-CSF with a `flower' initial curve, i.e., a nonconvex curve  parametrized by
\begin{equation*}
\bX(\rho)=((2+\cos(12\pi\rho))\cos(2\pi\rho),(2+\cos(12\pi\rho))\sin(2\pi\rho)),\quad \rho\in \mathbb{I}=[0,1].
\end{equation*}
 Figure \ref{Fig:flower_shape_evo} (bottom row) demonstrates that when approximating  the integration polygon $\widetilde{\Gamma}^{m+1}$ by extrapolation formulae \eqref{exbdf3}, the evolution becomes unstable even at very early stage  $t=0.05$, eventually leading to a breakdown of the algorithm (see Figure \ref{Fig:flower_shape_evo}(d2)).

\begin{figure}[htpb]
\centering
\includegraphics[width=5.1in,height=2.5in]{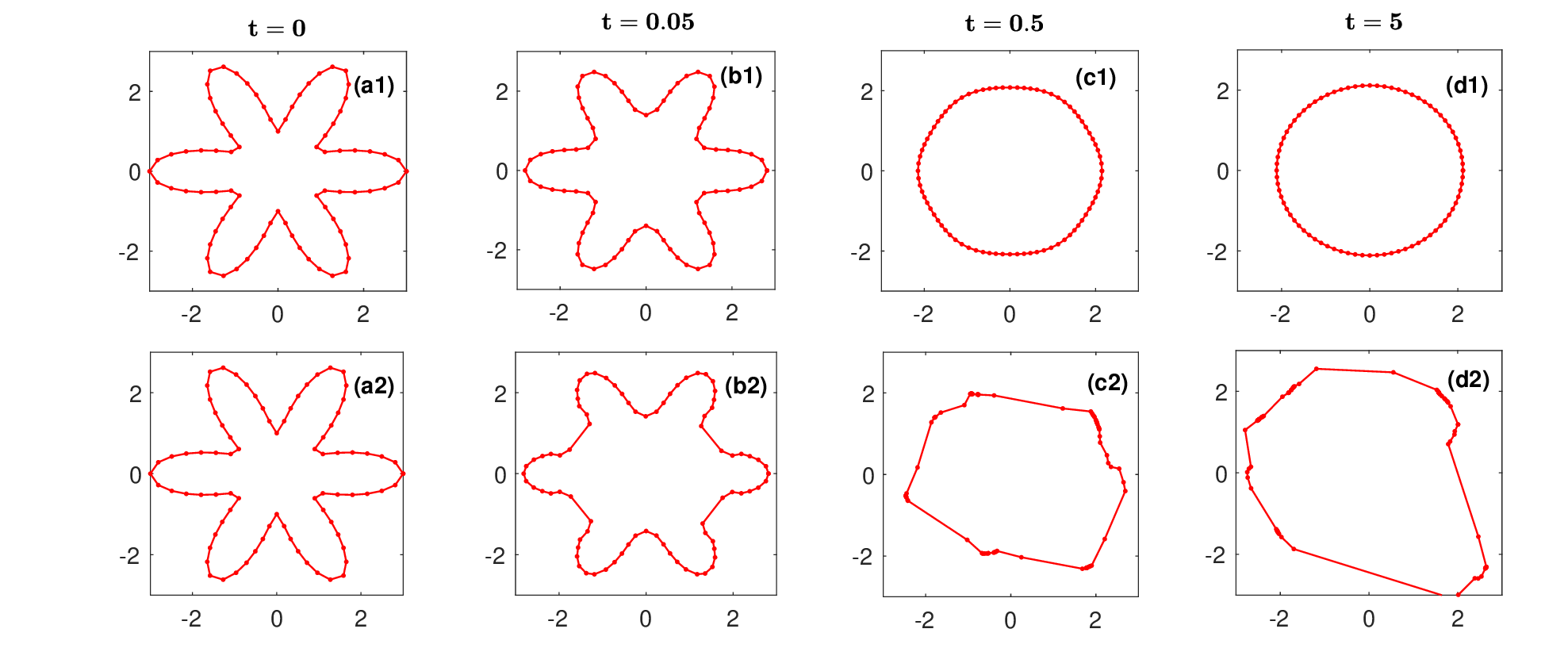}
\setlength{\abovecaptionskip}{2pt}
\caption{Evolution of AP-CSF for the  `flower' initial curve by using BGN/BDF3 algorithm with various choices of predictions  $\widetilde{\bX}^{m+1}$. Top row: Algorithm \ref{BGN/BDF3 algorithm}, i.e., approximating the prediction polygon by lower-order BGN/BDF2 scheme; Bottom row: approximating the prediction polygon by using extrapolation formulae \eqref{exbdf3}. The parameters are chosen as $N = 80$ and $\tau=1/160$.}
\label{Fig:flower_shape_evo}
\end{figure}

\begin{figure}[htpb]
\centering
\includegraphics[width=5.1in,height=2.8in]{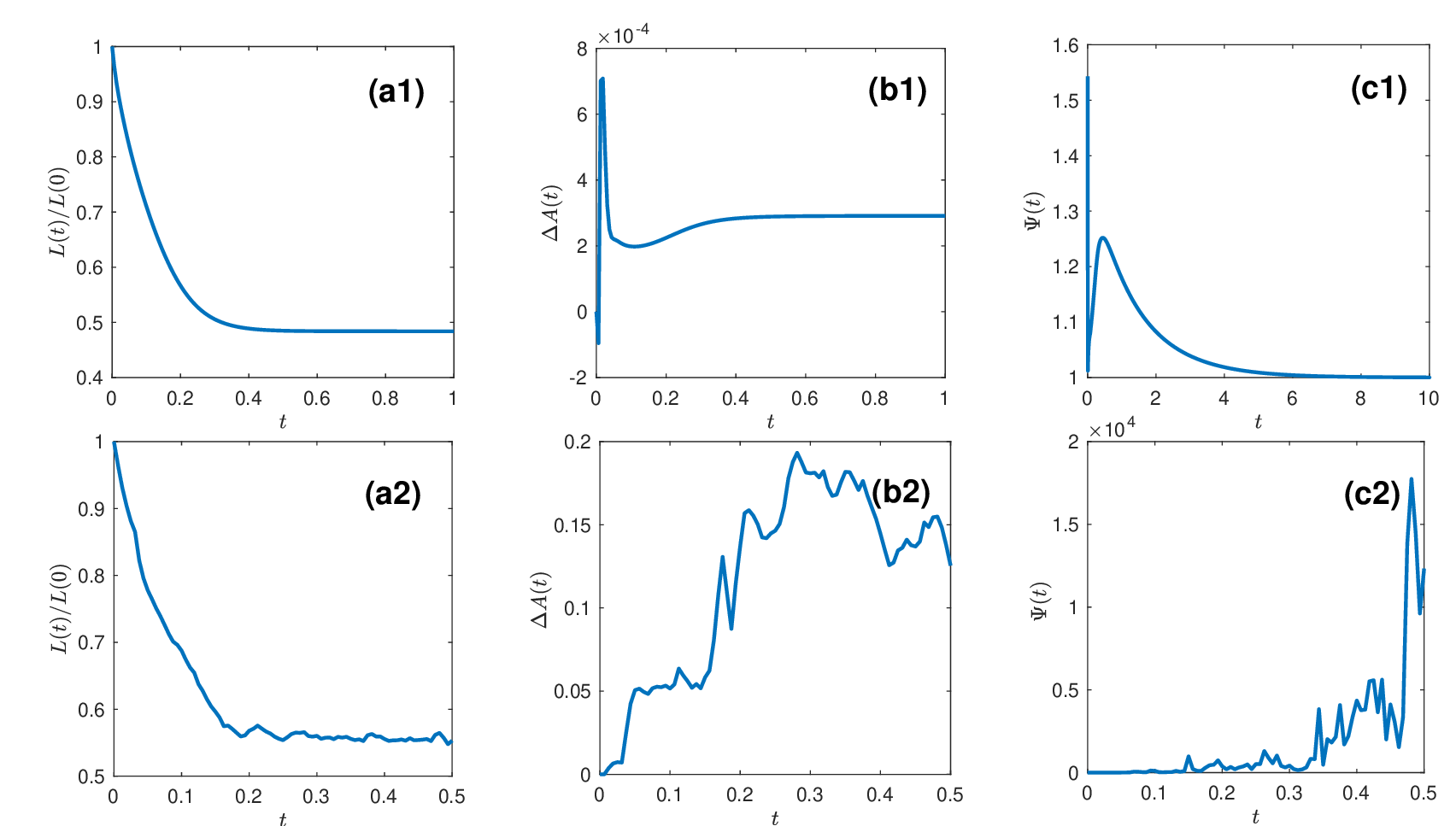}
\setlength{\abovecaptionskip}{2pt}
\caption{Evolution of geometric quantities of AP-CSF for `flower' initial shape by using BGN/BDF3 algorithm: (a1)-(a2) the normalized perimeter; (b1)-(b2) the normalized area loss; (c1)-(c2) the mesh distribution function $\Psi(t)$. Top row: the prediction polygon $\widetilde{\bX}^{m+1}$ is obtained by lower-order BGN/BDF2 scheme. Bottom row: $\widetilde{\bX}^{m+1}$  is obtained by the extrapolation formulae \eqref{exbdf3}. The parameters are chosen as $N = 80$ and  $\tau=1/160$.}
\label{Fig:flower_geo_evo}
\end{figure}

We further compare the evolution of geometric quantities in Figure \ref{Fig:flower_geo_evo}. Figure \ref{Fig:flower_geo_evo}(a1)-(b1)  demonstrate that our BGN/BDF3 scheme effectively preserves the perimeter-decreasing property of AP-CSF and the error of area is very small. Most importantly, Figure \ref{Fig:flower_geo_evo}(c1)  illustrates the long-time asymptotic mesh equidistribution property of our BGN/BDF3 algorithm. In comparison, when approximating the integration polygon by the extrapolation formulae, we observe from  Figure \ref{Fig:flower_geo_evo} (bottom row) that the mesh ratio becomes extremely large after some time, ultimately leading to the instability of BGN/BDF3 algorithm. This highlights the essential role of our treatment for approximations of the prediction polygon $\widetilde{\Gamma}^{m+1}$  through the classical BGN1 scheme or lower-order BGN/BDF$k$ schemes, which ensures the numerical stability and long-time mesh equidistribution.

\end{remark}

\smallskip

\begin{example}[Extension to G-MCF]
We extend the BGN/BDF$k$ schemes to G-MCF. For the convergence rate test, the initial shape is chosen as  a unit circle, the parameters are chosen as $N=5000$, $T=0.25$.

\end{example}

\smallskip

The construction of  BGN/BDF$k$ schemes for G-MCF is similar with the AP-CSF case. We first rewrite the coupled equations as
\begin{equation}\label{G-MCF:Coupled equation}
\begin{split}
	\p_t \mathbf{X}\cdot \mathbf{n} &=-\beta \kappa^\alpha,\\
  	\kappa \mathbf{n}&=-\p_{ss}\mathbf{X}.
\end{split}
\end{equation}
The corresponding BGN1 scheme and BGN/BDF$k$ schemes adjust the first equation in \eqref{CSF:BGN1} and \eqref{CSF:BDFk} to implicit terms as
\begin{align*}
	\l(\frac{\mathbf{X}^{m+1}-\mathbf{X}^{m}}{\tau},\varphi^h \mathbf{n}^{m} \r)^h_{\Gamma^{m}}+\l( \beta (\kappa^{m+1})^\alpha ,\varphi^h \r)_{\Gamma^m}^h &=0,\\
	\l(\frac{a \mathbf{X}^{m+1}-\mathbf{\widehat{X}}^{m}}{\tau},\varphi^h \mathbf{\widetilde{n}}^{m+1} \r)^h_{\widetilde{\Gamma}^{m+1}}+\l(  \beta (\kappa^{m+1})^\alpha,\varphi^h \r)_{\widetilde{\Gamma}^{m+1}}^h&=0,
\end{align*}
respectively, where $a,\widehat{\bX}^m$ are defined similarly. All of  above schemes can be efficiently solved using  Newton's iteration method \cite{BGN07B,Pei-Li}.

\begin{figure}[htpb]
\centering
\includegraphics[width=5.1in,height=2.8in]{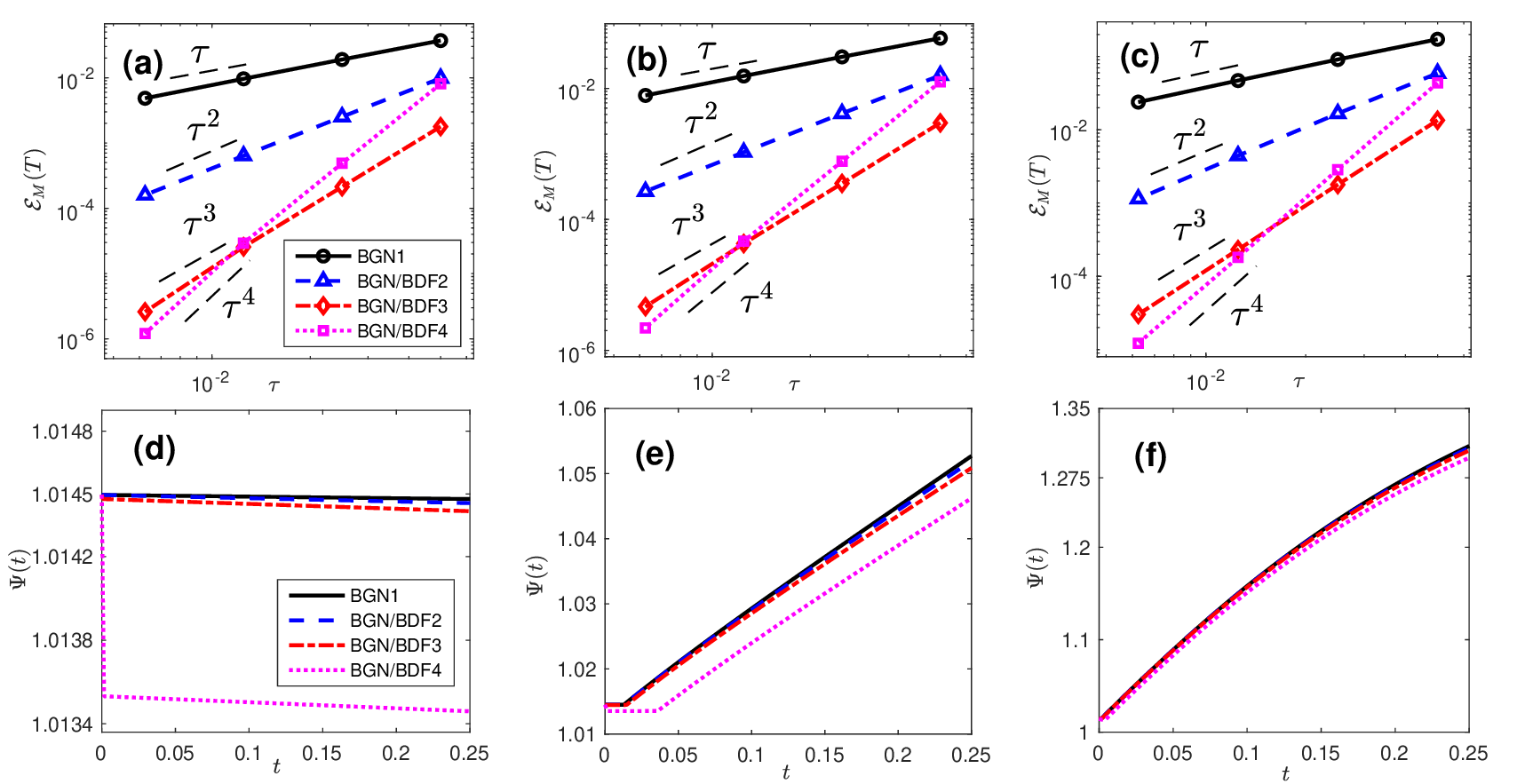}
\setlength{\abovecaptionskip}{2pt}
\caption{(Top row) Log-log plot of the manifold distance errors for solving G-MCF with a unit circle as its initial shape at time $T=0.25$ under different choices of $(\alpha,\beta)$: (a) $\beta=1$, $\alpha=1/3$, (b) $\beta=1$, $\alpha=1/2$,  (c) $\beta=-1$, $\alpha=-1$;  (Bottom row) the evolution of the corresponding mesh distributional function $\Psi(t)$ under different choices of $(\alpha,\beta)$: (d) $\beta=1$, $\alpha=1/3$, (e) $\beta=1$, $\alpha=1/2$,  (f) $\beta=-1$, $\alpha=-1$, where $N=640$ and $\tau=1/1280$.}
\label{Fig:G-MCF}
\end{figure}

In order  to test the temporal convergence rate for various choices of $\alpha$ and $\beta$, we use a unit circle as an initial data. The true solution is given by
\begin{align*}
	\bX_{\mathrm{true}}(\rho,t)= \begin{cases}
		(1-(\alpha+1)t)^{\frac{1}{\alpha+1}}(\cos(2\pi\rho),\sin(2\pi\rho)), \quad &\beta=1,\quad 0<\alpha\neq 1, \\
		e^t(\cos(2\pi\rho),\sin(2\pi\rho)),\quad &\beta=-1,\quad \alpha=-1,
	\end{cases}
\end{align*}
where  $\rho\in \mathbb{I}$ and $t\in [0,T]$. We use a sufficiently large number of  grid points $N=5000$, and we observe the expected convergence order for different cases, as depicted  in Figure \ref{Fig:G-MCF}(a)-(c).	Moreover, Figure \ref{Fig:G-MCF}(d)-(f) demonstrate that the BGN/BDF$k$ algorithms maintain the good mesh quality comparable to the classical BGN1 scheme, regardless of the specific settings of $(\alpha,\beta)$.



\smallskip

\begin{example}[Extension to WF]
We extend the BGN/BDF$k$ schemes to the fourth-order WF. The initial shape is chosen as a unit circle which yields the exact solution \cite{BGN08C}
\[
\bX_{\mathrm{true}}(\rho,t)= (1+2t)^{\frac{1}{4}}(\cos(2\pi\rho),\sin(2\pi\rho)),\quad \rho\in \mathbb{I},\quad t>0.
\]
\end{example}

\smallskip

Similar as the AP-CSF and G-MCF cases, we first reformulate the coupled equations as
\begin{equation}\label{WF:Coupled equation}
\begin{split}
	\p_t \mathbf{X}\cdot \mathbf{n} &=\p_{ss}\kappa-\frac{1}{2}\kappa^3,\\
  	\kappa \mathbf{n}&=-\p_{ss}\mathbf{X}.
\end{split}
\end{equation}
The corresponding BGN1 scheme \cite{BGN08C} and BGN/BDF$k$ schemes adjust the first equation in \eqref{CSF:BGN1} and \eqref{CSF:BDFk} to
\begin{align*}
	\l(\frac{\mathbf{X}^{m+1}-\mathbf{X}^{m}}{\tau},\varphi^h \mathbf{n}^{m} \r)^h_{\Gamma^{m}}+\l( \p_s\kappa^{m+1},\p_s\varphi^h \r)_{\Gamma^m} &=-\frac{1}{2}\l( (\kappa^{m+1})^3,\varphi^h \r)_{\Gamma^m}^h,\\
	\l(\frac{a \mathbf{X}^{m+1}-\mathbf{\widehat{X}}^{m}}{\tau},\varphi^h \mathbf{\widetilde{n}}^{m+1} \r)^h_{\widetilde{\Gamma}^{m+1}}+\l( \p_s\kappa^{m+1},\p_s\varphi^h \r)_{\widetilde{\Gamma}^{m+1}}&=-\frac{1}{2}\l((\kappa^{m+1})^3,\varphi^h \r)_{\widetilde{\Gamma}^{m+1}}^h,
\end{align*}
respectively.

Figure \ref{Fig:WF}(a)-(b) display the convergence and the evolution of the mesh ratio $\Psi(t)$, respectively, which suggests that the BGN/BDF$k$ scheme converges at the $k$-th order in terms of manifold distance and it maintains the good mesh quality throughout the simulation of the fourth-order WF.

\begin{figure}[htpb]
\centering
\includegraphics[width=5.1in,height=1.8in]{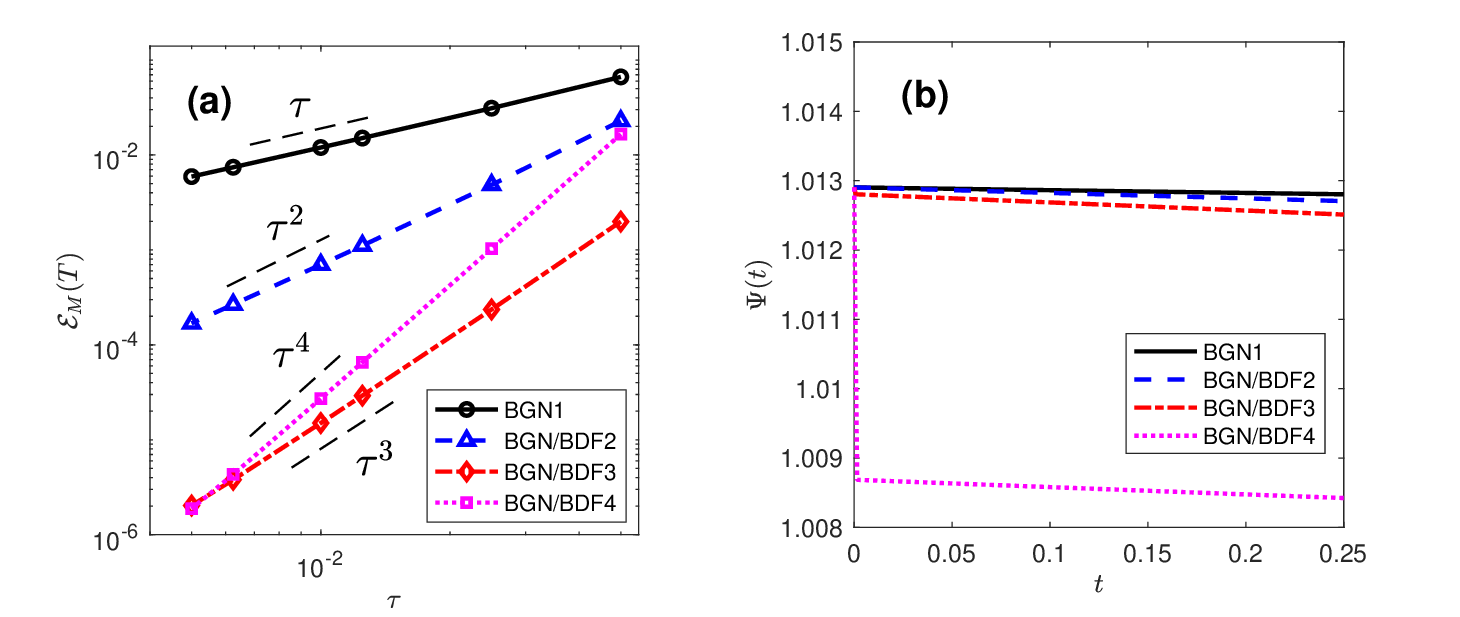}
\setlength{\abovecaptionskip}{-0pt}
\caption{(a) Log-log plot of the numerical  errors for solving the WF with unit circle  initial shape at time $T=0.25$. (b) The evolution of the mesh distribution function $\Psi(t)$, where the parameters are chosen as $N=640$ and $\tau=1/1280$.}
\label{Fig:WF}
\end{figure}

\smallskip

\subsection{Surface evolution}
In this subsection, we mainly consider the following geometric flows of surface evolution in three-dimensional space:

\begin{itemize}
	\item Mean curvature flow (MCF), which is a second-order geometric flow defined by
	\begin{equation}
		\cV = -\mathcal{H}\, \bn,
	\end{equation}
	where $\mathcal{H}$ is the mean curvature of hypersurface;
	\item Surface diffusion flow (SDF), which is a fourth-order geometric flow defined by
	\begin{equation}\label{SDF}
		\cV = \Delta_\Gamma \mathcal{H}\, \bn,
	\end{equation}
	where $\Delta_\Gamma$ is the surface Laplacian.
\end{itemize}

\smallskip

Given an initial surface $\Gamma(0)$, we use the  MATLAB package  \texttt{DistMesh} \cite{Persson} to give an initial triangulation polyhedron $\Gamma^0$ with good mesh quality. It is worth mentioning that the manifold distance \eqref{Manifold distance} can be easily extended to the three-dimensional case, and we still use this shape metric to measure the difference between two polyhedrons. We point out that other shape metrics such as Hausdorff distance \cite{Jiang23B,Bao-Zhao} can also effectively characterize the convergence rates of our schemes. The numerical error and convergence order are similarly defined as \eqref{eqn:errordef1}.

We further investigate the evolution of the following geometric quantities: (1) the relative volume loss $\Delta V(t)$, (2) the normalized surface area $S(t)/S(0)$, which are defined respectively as follows for $m\ge 0$:
\[
\Delta V(t)|_{t=t_m}=\frac{V^m-V^0}{V^0},\quad \l.\frac{S(t)}{S(0)}\r|_{t=t_m}=\frac{S^m}{S^0},
\]
where $V^m$ is the volume enclosed by the polyhedron  determined by $\bX^m$, and $S^m$ represents the surface area of the polyhedron. To evaluate the mesh quality of a polyhedron, we introduce two  mesh distribution functions $r_h(t)$ and $r_a(t)$ \cite{BGN08B,Hu2022} defined as
\begin{align*}
	&\l.r_h(t)\r|_{t=t_m}
	:=\frac{\max_{j}\max \{\|\mathbf{q}^m_{j_1}-\mathbf{q}^m_{j_2}\|,\|\mathbf{q}^m_{j_2}-\mathbf{q}^m_{j_3}\|,\|\mathbf{q}^m_{j_3}-\mathbf{q}^m_{j_1}\| \}}{\min_{j}\min \{\|\mathbf{q}^m_{j_1}-\mathbf{q}^m_{j_2}\|,\|\mathbf{q}^m_{j_2}-\mathbf{q}^m_{j_3}\|,\|\mathbf{q}^m_{j_3}-\mathbf{q}^m_{j_1}\|\}},\\
&\l.r_a(t)\r|_{t=t_m}
	:=\frac{\max_j|\sigma_j^m|}{\min_j|\sigma_j^m|}.
\end{align*}

\smallskip

\begin{example}[Extension to MCF]
The BGN/BDF$k$ schemes can be extended straightforwardly to the mean curvature flow in $\mathbb{R}^3$. It suffices to replace the curvature $\kappa^{m+1}$ in \eqref{CSF:BDFk} by the mean curvature $\mathcal{H}^{m+1}$.
\end{example}

For the test of convergence order, we choose the initial surface as a unit sphere,
which remains as a sphere with radius $R(t)$ given by
\[
R(t)=\sqrt{1-4t},\quad t\in \l[0,1/4\r).
\]

\begin{figure}[htpb]
\centering
\includegraphics[width=5.1in,height=1.5in]{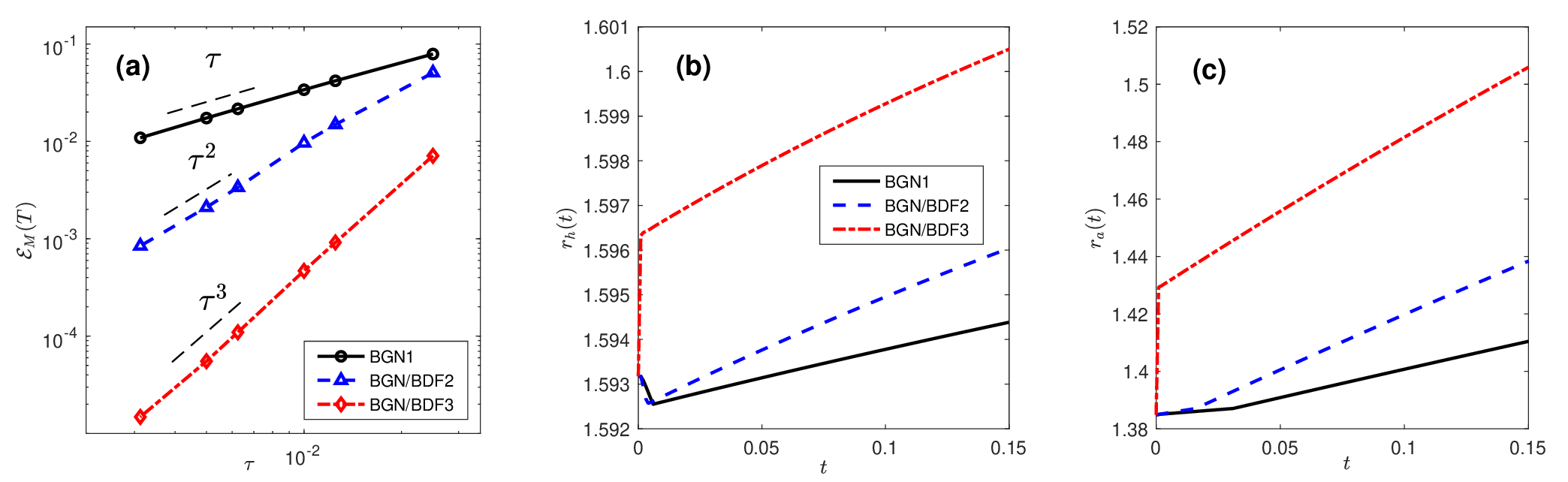}
\setlength{\abovecaptionskip}{-0pt}
\caption{(a) Log-log plot of the manifold distance errors for solving the MCF until $T=0.05$ with unit sphere initial shape, where the spatial parameters are chosen as $(J,K)=(93608,46806)$. The evolution of the mesh distribution functions: (b) $r_h(t)$ and (c) $r_a(t)$, where the parameters are chosen as $(J,K)=(14888,7446)$ and $\tau=1/1000$.}
\label{Fig:MCF}
\end{figure}

For the initial triangulation, we fix the fine mesh size $(J,K)=(93608,46806)$, where $J$ and $K$ represent the number of triangles and vertices of polyhedron, respectively.
Figure \ref{Fig:MCF}(a) shows a log-log plot of the manifold errors at time $T=0.05$ for BGN1 scheme and BGN/BDF$k$ schemes with $2\le k\le 3$, which demonstrates the expected convergence rate. From Figure \ref{Fig:MCF}(b)-(c), we observe that the two mesh distribution functions $r_h(t)$ and $r_a(t)$ remain below $2$ throughout the evolution, indicating the good mesh quality of our proposed BGN/BDF$k$ algorithms. We note that for BGN/BDF3 algorithm, the two mesh distribution functions increase rapidly at the first step. The reason is that we need a fine time step to initiate BGN/BDF3 scheme to ensure the third-order accuracy (see Algorithm \ref{BGN/BDF3 algorithm}), however, as reported in the recent paper~\cite{Duan2023}, the classical BGN1 scheme for surface evolution may result in mesh clustering for very fine time step sizes.


\begin{figure}[htpb]
\hspace{-10mm}
\includegraphics[width=5.5in,height=3.1in]{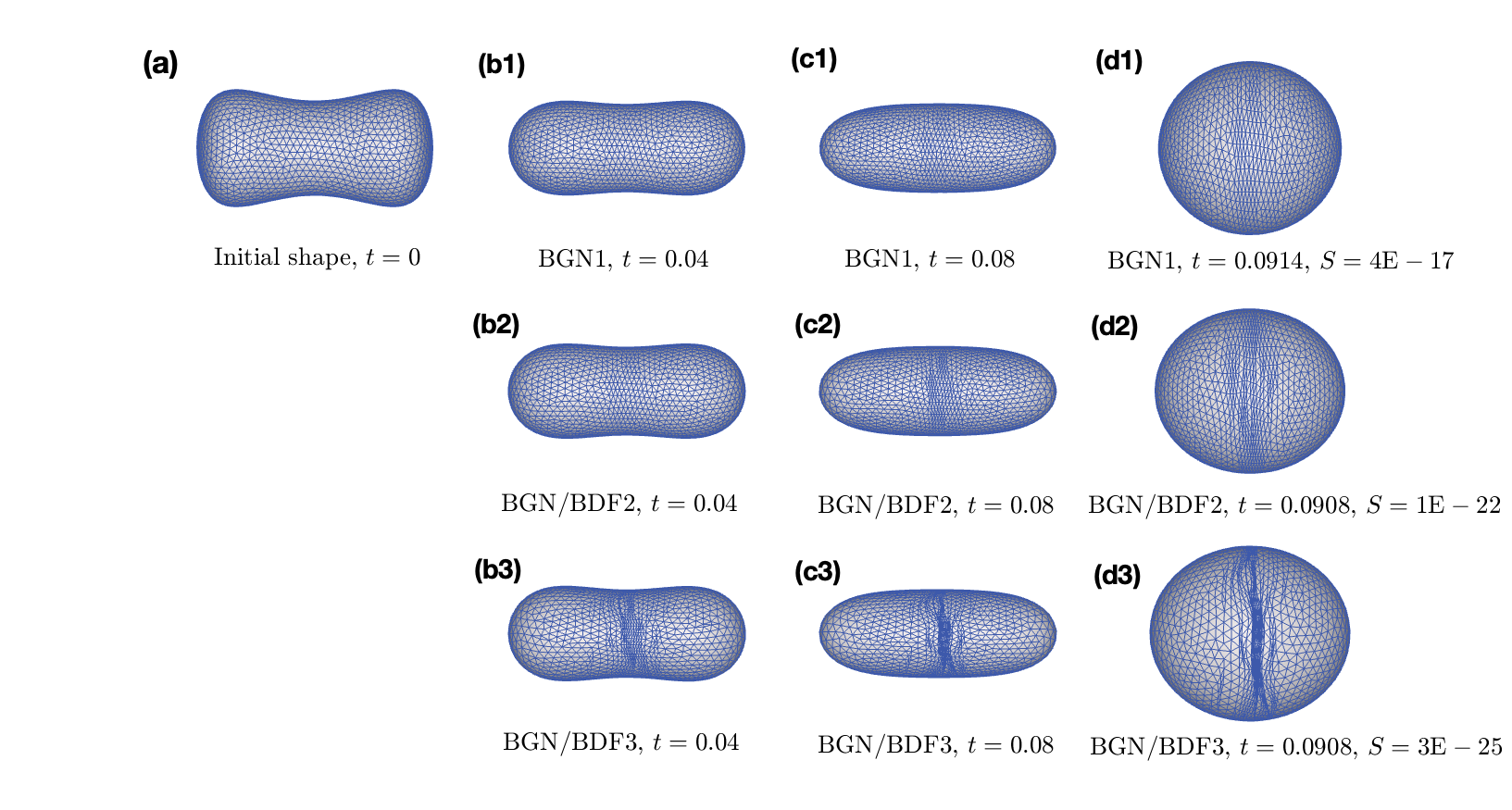}
\setlength{\abovecaptionskip}{-5pt}
\caption{Evolution  of MCF by BGN1 scheme (first row), BGN/BDF2 scheme (second row) and BGN/BDF3 scheme (third row) starting from the first dumbbell with a fat waist (note that the images are scaled). The initial surface is triangulated into $3604$ triangles with $1804$ vertices and the time step $\tau=1/10000$.}
\label{Fig:EVO_dumbbell1_re}
\end{figure}

\begin{figure}[htpb]
\hspace{-8mm}
\includegraphics[width=5.5in,height=2.4in]{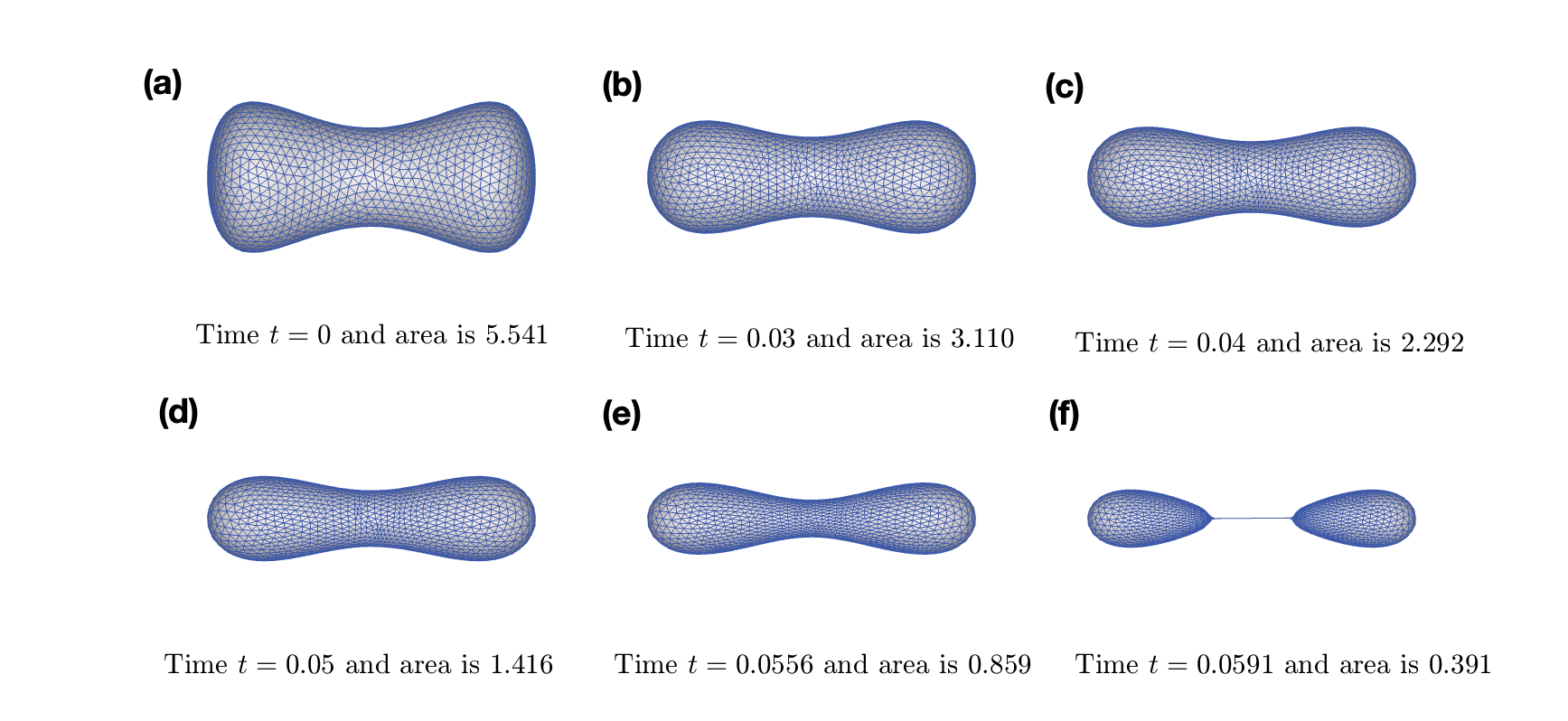}
\setlength{\abovecaptionskip}{-5pt}
\caption{Evolution  of MCF by BGN/BDF2 scheme starting from the first dumbbell with a thin waist (the images are scaled). The initial surface is triangulated into $3276$ triangles with $1640$ vertices and the time step $\tau=1/10000$. }
\label{Fig:EVO_dumbbell2}
\end{figure}

For the evolution test, we apply BGN/BDF$k$ schemes  to two benchmark dumbbell examples \cite{Elliott-Fritz,Duan2023}. The initial surface is a dumbbell-shape surface with a fat waist given by the following parametrization
\be\label{exdd1}
\bX(\theta,\varphi)=\begin{pmatrix}
	\cos\varphi\\
	(0.6\cos^2\varphi+0.4)\cos\theta\sin\varphi\\
	(0.6\cos^2\varphi+0.4)\sin\theta\sin\varphi
\end{pmatrix},\quad \theta\in [0,2\pi),\quad \varphi\in[0,\pi],
\ee
or a dumbbell shape with a thin  waist parameterized by
\be\label{exdd2}
\bX(\theta,\varphi)=\begin{pmatrix}
	\cos\varphi\\
	(0.7\cos^2\varphi+0.3)\cos\theta\sin\varphi\\
	(0.7\cos^2\varphi+0.3)\sin\theta\sin\varphi
\end{pmatrix},\quad \theta\in [0,2\pi),\quad \varphi\in[0,\pi].
\ee
The numerical simulations are presented in Figure \ref{Fig:EVO_dumbbell1_re} and Figure \ref{Fig:EVO_dumbbell2}, respectively. We can clearly observe that MCF evolves  the first  dumbbell shape to a round point (c.f. Figure \ref{Fig:EVO_dumbbell1_re}(d1), (d2) and (d3)), and  develops a neck pinch singularity for the second  dumbbell shape in finite time  (c.f. Figure \ref{Fig:EVO_dumbbell2}(f)). We observe that the BGN/BDF$k$ algorithms  achieve similar mesh quality as the classical BGN1 scheme. The mesh is slightly distorted in Figure \ref{Fig:EVO_dumbbell1_re}(d3) since we use a fine time step size to predict the polyhedron at the first step (cf. Algorithm \ref{BGN/BDF3 algorithm}), which may result in mesh distortion for surface evolution as reported in \cite{Duan2023}. Nevertheless, we have successfully obtained the blow-up times, that is $t=0.0908$ for the first dumbbell shape and $t=0.0591$ for the second dumbbell shape, respectively. These results are comparable to those reported  in \cite[Section 7.2 and 7.3]{Elliott-Fritz} and \cite[Example 3.2]{Duan2023}.

\begin{figure}[htpb]
\centering
\includegraphics[width=5.1in,height=2.8in]{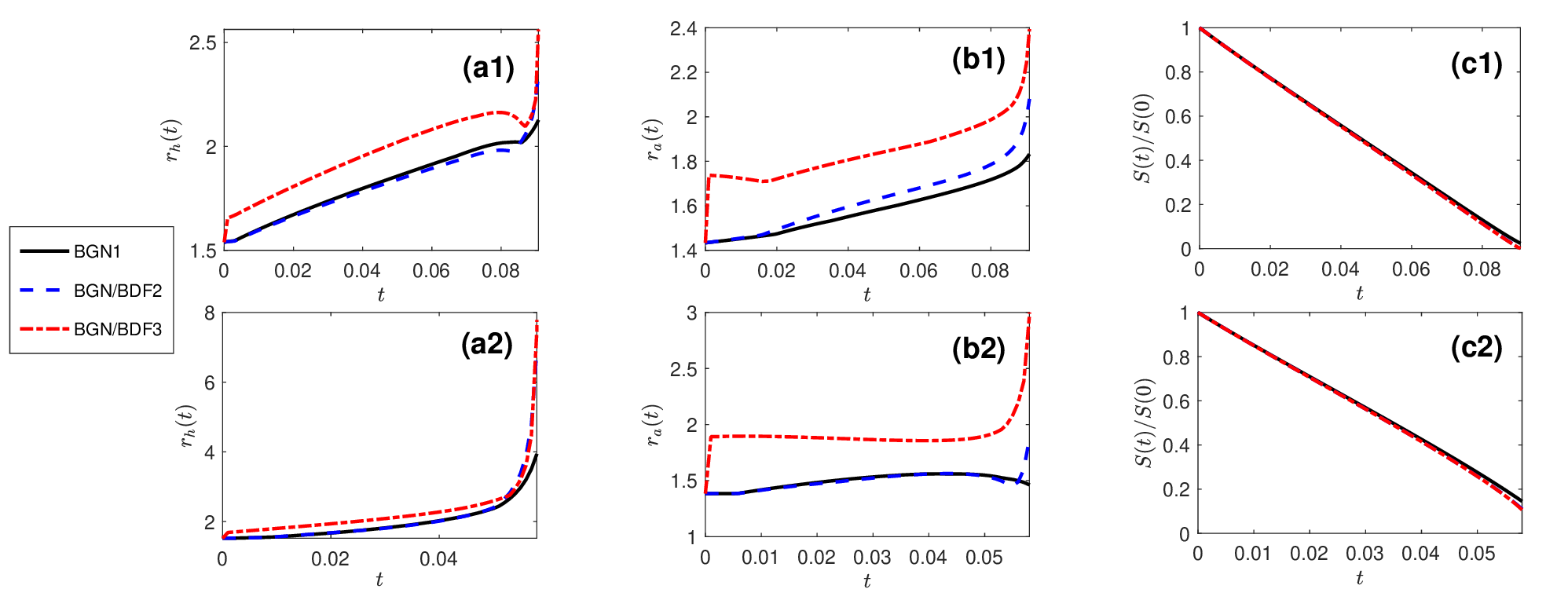}
\setlength{\abovecaptionskip}{0pt}
\caption{Evolution of geometric quantities in MCF for two initial dumbbell shapes by using BGN/BDF$k$ algorithms: the first dumbbell example \eqref{exdd1} (top row); the second dumbbell example \eqref{exdd2} (bottom row). The mesh distribution functions $r_h(t)$: (a1)-(a2); and $r_a(t)$: (b1)-(b2). The normalized surface area: (c1)-(c2). The time step is chosen  as  $\tau=1/1000$.}
\label{EVO_Geo_dumbbell}
\end{figure}

The evolution of the two mesh distribution functions $r_h(t)$, $r_a(t)$,  and the normalized surface area are plotted in Figure \ref{EVO_Geo_dumbbell}. As clearly shown in Figure  \ref{EVO_Geo_dumbbell}(a)-(b), our high-order schemes exhibit  similar mesh behavior as the classical BGN1 scheme before the blow-up time. Moreover, Figure \ref{EVO_Geo_dumbbell}(c1)-(c2) demonstrate that our methods preserve the geometric property (i.e., the decreasing surface area) of MCF very well.

\smallskip

\begin{example}[Extension to SDF]
 We investigate the performance of BGN/BDF$k$ schemes when applied to SDF in this example.
\end{example}

\smallskip

Specifically, the schemes for SDF can be similarly derived as the MCF case and it suffices to adjust the first equation as
\begin{align*}
	\l(\frac{\mathbf{X}^{m+1}-\mathbf{X}^{m}}{\tau},\varphi^h \mathbf{n}^{m} \r)^h_{\Gamma^{m}} +\l(\nabla_{\Gamma^m}  \mathcal{H}^{m+1} ,\nabla_{\Gamma^m}\varphi^h \r)_{\Gamma^m}&=0,\\
	\l(\frac{a \mathbf{X}^{m+1}-\mathbf{\widehat{X}}^{m}}{\tau},\varphi^h \mathbf{\widetilde{n}}^{m+1} \r)^h_{\widetilde{\Gamma}^{m+1}}+\l( \nabla_{\Gamma^m} \mathcal{H}^{m+1},\nabla_{\Gamma^m}\varphi^h \r)_{\widetilde{\Gamma}^{m+1}}&=0.
\end{align*}

Instead of conducting a convergence order test, we demonstrate the superiority of our high-order schemes by examining the relative volume loss $\Delta V(t)$. The initial surface is a $2:1:1$ ellipsoid defined by the equation  $x^2/4+y^2+z^2=1$, and we set the triangulation parameters as $(J,K) = (24952, 12478)$.  Figure \ref{EVO_Geo_ell_SDF} illustrates that all methods  maintain the geometric properties of SDF, i.e., the decrease of the surface area and the conservation of the volume enclosed by the surface. Additionally, Figure \ref{EVO_Geo_ell_SDF}(a2), (b2) and (c2) highlight that our high-order algorithms result in significantly smaller volume loss
compared to BGN1 scheme \cite{BGN08B}, thus showcasing the accuracy of our proposed methods.

\begin{figure}[htpb]
\centering
\includegraphics[width=5.1in,height=2.8in]{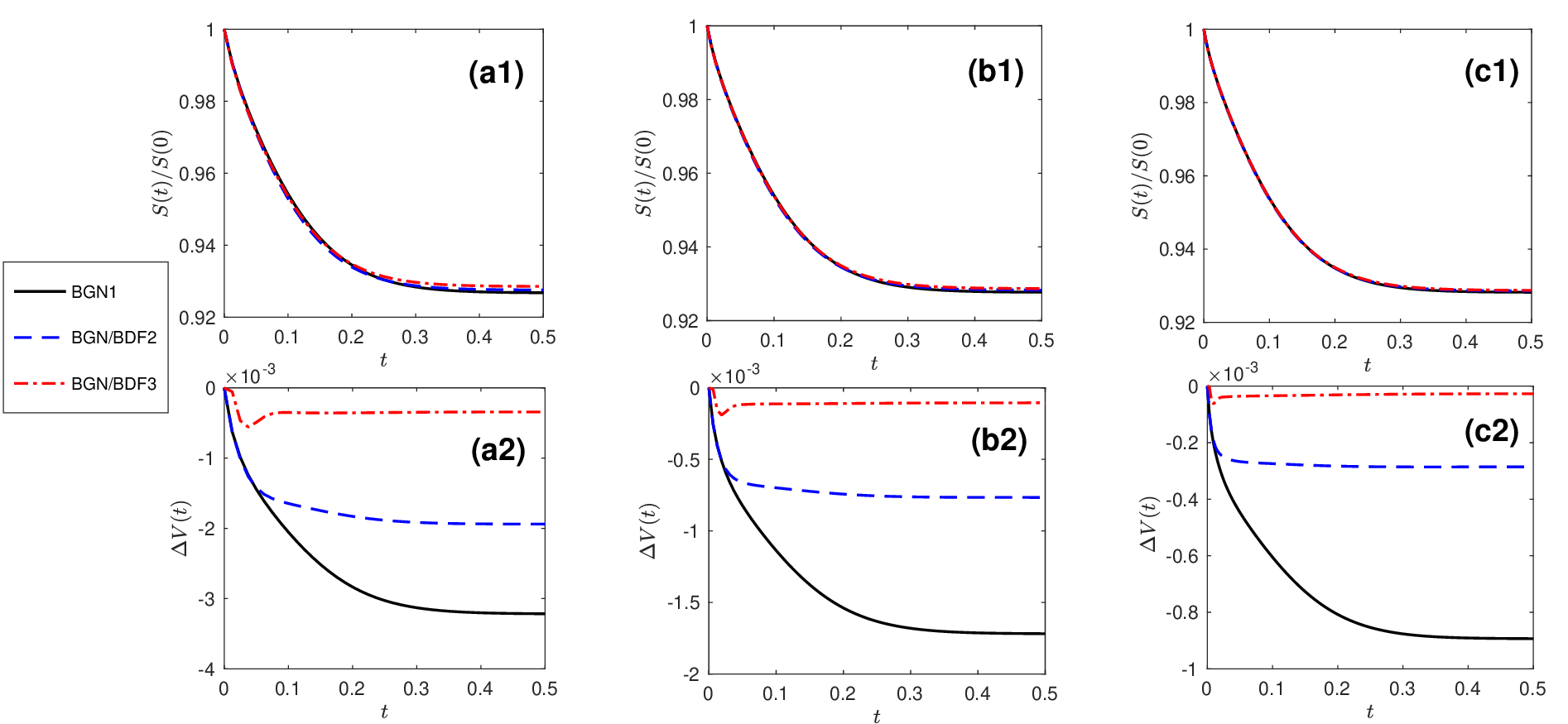}
\setlength{\abovecaptionskip}{0pt}
\caption{Evolution of geometric quantities for SDF applied to an initial ellipsoid using BGN/BDF$k$ algorithms. Top row: the normalized surface area; Bottom row: the relative volume loss. The triangulation parameters are set as $(J,K) = (24952, 12478)$ and the time step is set as $\tau=1/160$ for (a1), (a2); $\tau=1/320$ for (b1), (b2); and $\tau=1/640$ for (c1), (c2). }
\label{EVO_Geo_ell_SDF}
\end{figure}

In Figure \ref{Fig:EVO_ellipsoid}, we present several evolution snapshots of the $2:1:1$ ellipsoid towards its equilibrium using BGN/BDF2 scheme with the surface triangulation parameters $(J,K)= (2780,1392)$. Furthermore, Figure \ref{Fig:EVO_Geo_ellipsoid} depicts the  comparison of the two mesh distribution functions $r_h(t)$ and $r_a(t)$ for both BGN1 and BGN/BDF$k$ schemes. It is worth noting that our high-order method successfully evolves the ellipsoid into a perfect sphere while maintaining good mesh quality, similar to the classical BGN scheme.

\begin{figure}[htpb]
\centering
\includegraphics[width=5.6in,height=2.8in]{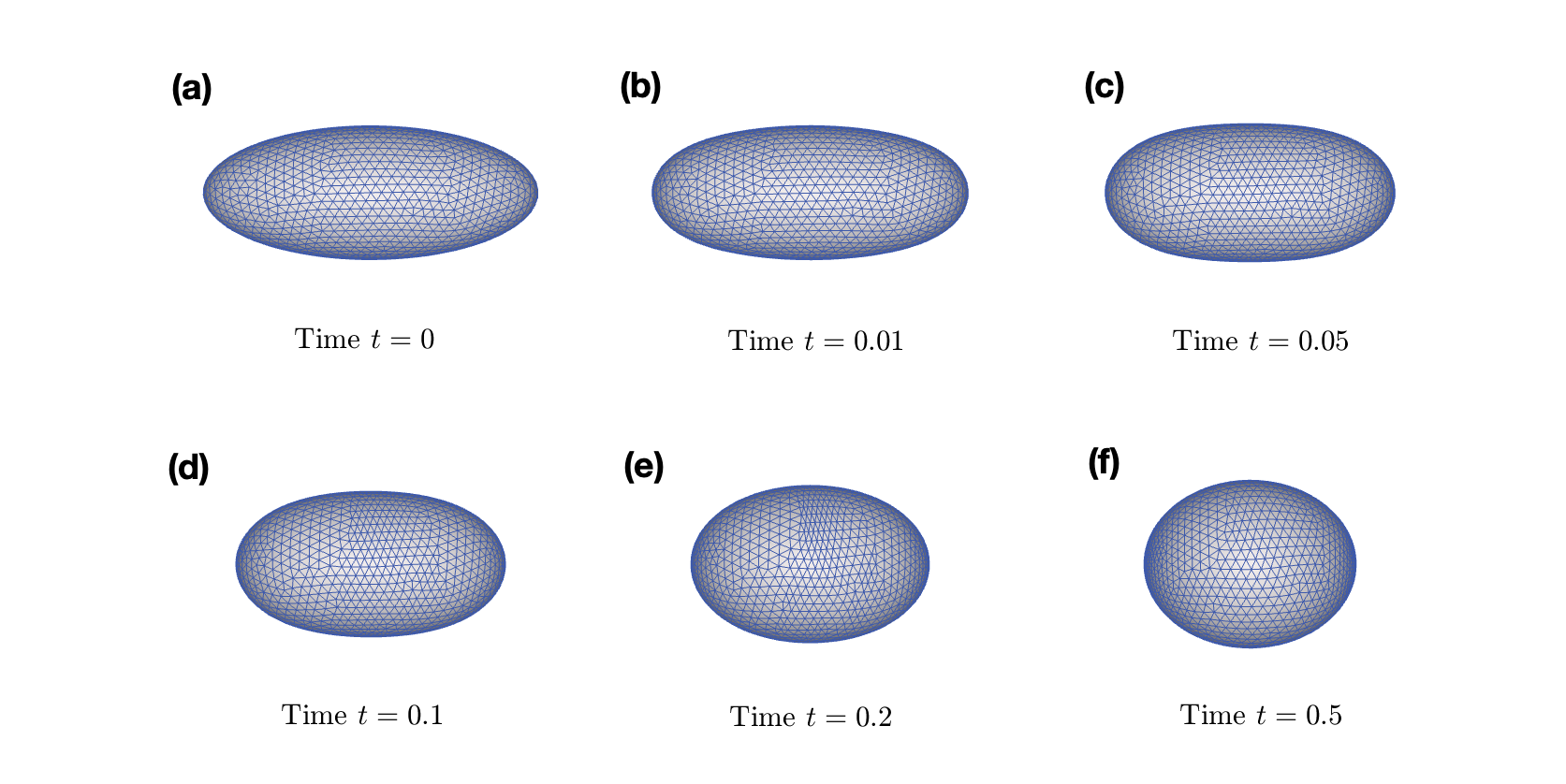}
\setlength{\abovecaptionskip}{-5pt}
\caption{Several evolution snapshots of SDF by BGN/BDF2 scheme, where the initial shape is chosen as the $2:1:1$ ellipsoid,  and
the triangulation parameters $(J,K)=(2780,1392)$ and  $\tau=10^{-4}$.}
\label{Fig:EVO_ellipsoid}
\end{figure}

\begin{figure}[htpb]
\centering
\includegraphics[width=5.1in,height=1.8in]{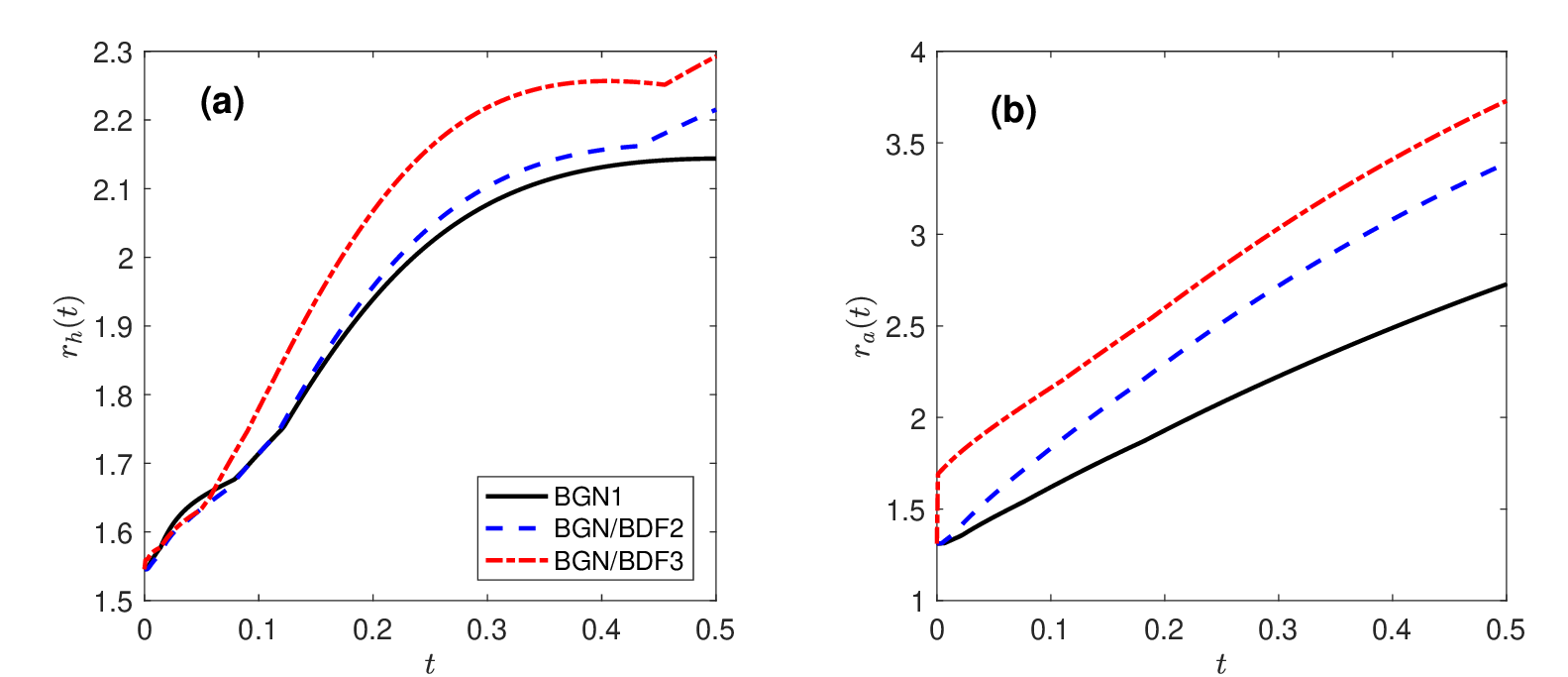}
\setlength{\abovecaptionskip}{5pt}
\caption{Evolution of mesh distribution functions of SDF for the $2:1:1$ ellipsoid: (a) $r_h(t)$, (b) $r_a(t)$, where $(J,K)=(2780,1392)$ and  $\tau=10^{-4}$.}
\label{Fig:EVO_Geo_ellipsoid}
\end{figure}

\begin{figure}[h!]
\centering
\includegraphics[width=5.1in,height=4.2in]{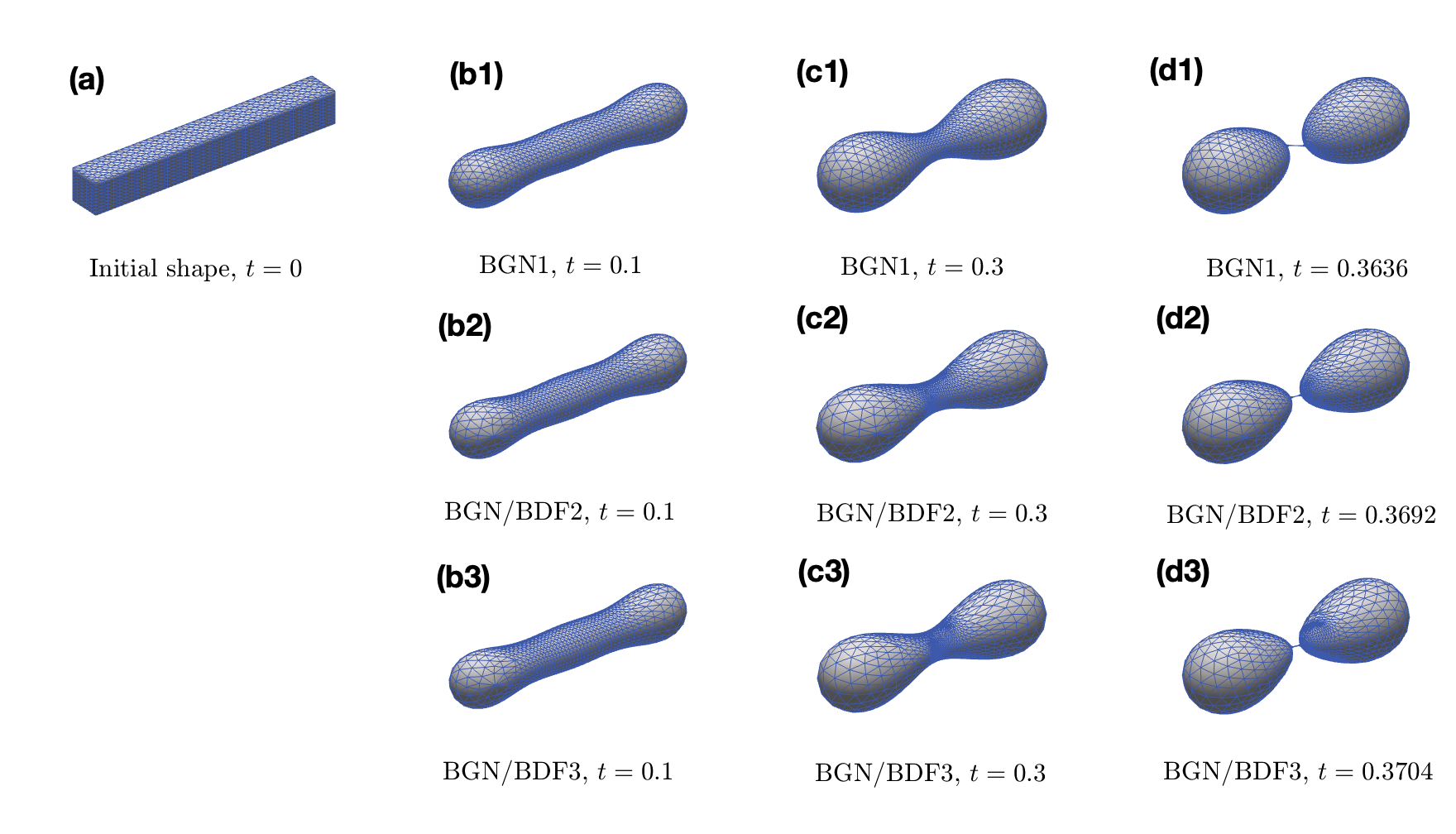}
\setlength{\abovecaptionskip}{-5pt}
\caption{Investigation of pinch-off time of an $8\times 1\times 1$ cuboid driven by SDF by using BGN1 scheme (first row), BGN/BDF2 algorithm (second row) and BGN/BDF3 algorithm (third row), where $(J,K)=(2600,1302)$ and $\tau=1/2500$.}
\label{Fig:EVO_cuboid}
\end{figure}

In the last experiment, we apply our BGN/BDF$k$ schemes to conduct the evolution of an $8\times 1\times 1$ cuboid driven by SDF. As a benchmark example, this cuboid  pinches off at finite time \cite{BGN08B,Bao-Zhao}. We select the triangulation parameters as $(J,K)=(2600,1302)$, and  the time step is fixed as  $\tau=1/2500$. The evolution is depicted in Figure \ref{Fig:EVO_cuboid}, where we observe that the pinch-off event occurs at time $t=0.3636$ for BGN1 scheme, $t=0.3692$ for  BGN/BDF2 algorithm and $t=0.3704$ for BGN/BDF3 algorithm, respectively. The pinch-off time of this benchmark cuboid example was reported as $t=0.369$ in \cite[Section 5.5]{BGN08B} using a tedious time adaptive method, and it was predicted as $t=0.370$ in \cite[Section 4.2]{Bao-Zhao}  by using  an implicit method. We emphasize that our  predicted pinch-off time agrees very well with the previous results~\cite{Bao-Zhao, BGN08B}, while requiring only to solve two or three linear systems at each time step. Furthermore, the monotonic increasing behavior of pinch-off times in Figure \ref{Fig:EVO_cuboid}(d1)-(d3) indicates that the high-order schemes provide a better prediction of the pinch-off time.

\smallskip

\section{Conclusion}
We have proposed a type of novel temporal high-order (second-order to fourth-order), parametric finite element methods based on the BGN formulation \cite{BGN07A,BGN07B,BGN20} for solving different types of geometric flows of curves and surfaces, including  CSF, AP-CSF, G-MCF, WF, MCF and SDF. Our approach is constructed based on the BGN formulation \cite{BGN07A,BGN07B,BGN20}, the backward differentiation formulae in time and linear finite element approximation in space. We carefully choose the prediction polygon $\widetilde{\Gamma}^{m+1}$ to ensure that the approximation errors of all quantities are at  $\mathcal{O}(\tau^k)$. The key to the success of BGN/BDF$k$ schemes is that $\widetilde{\Gamma}^{m+1}$ should be given by solving lower-order BGN/BDF$k$ schemes, instead of standard extrapolation, to maintain the mesh quality, which is very essential for the simulation of geometric flows.  Extensive numerical experiments demonstrate the expected, high-order accuracy and improved performance compared to the classical BGN scheme.

However, designing high-order schemes that preserve the geometric structure, specifically reducing the perimeter and conserving the enclosed area at the discrete level, remains challenging for certain geometric flows such as AP-CSF and SDF, which is our future work.

\section*{Acknowledgments}
This work is supported by the Center of High Performance Computing, Tsinghua University. The numerical calculations in this paper were also partially performed  on the supercomputing system in the Supercomputing Center of Wuhan University.


%
%

\end{document}